\documentclass[11pt, reqno]{amsart}

\usepackage{amsfonts}
\usepackage{amssymb}
\usepackage{amscd}
\usepackage{microtype}
\usepackage{graphicx}
\usepackage{hyperref}
\usepackage{slashed}
\usepackage{fullpage}

\usepackage{etoolbox}
\makeatletter
\let\ams@starttoc\@starttoc
\makeatother
\usepackage[parfill]{parskip}
\makeatletter
\let\@starttoc\ams@starttoc
\patchcmd{\@starttoc}{\makeatletter}{\makeatletter\parskip\z@}{}{}
\makeatother

\newtheorem{thm}{Theorem}[section]
\newtheorem{cor}[thm]{Corollary}
\newtheorem{lem}[thm]{Lemma}

\theoremstyle{remark}
\newtheorem{rem}[thm]{Remark}
\theoremstyle{definition}

\newtheorem{examp}[thm]{Example}
\newcommand{\N}{\mathbb{N}}

\newcommand{\R}{\mathbb{R}}
\renewcommand{\r}{\mathbf{r}}

\newcommand{\cT}{T}

\renewcommand{\Im}{\mathrm{Im}}

\DeclareMathOperator{\Tr}{Tr}
\DeclareMathOperator{\chg}{chg}
\begin{document}
\title{On the unconditional uniqueness of solutions to the infinite radial
Chern-Simons-Schr\"{o}dinger hierarchy}
\author{Xuwen Chen}
\address{Brown University}
\email{chenxuwen@math.brown.edu}
\urladdr{http://www.math.brown.edu/\symbol{126}chenxuwen/}
\author{Paul Smith}
\address{University of California, Berkeley}
\email{smith@math.berkeley.edu}
\urladdr{http://math.berkeley.edu/\symbol{126}smith/}
\thanks{The second author was supported by NSF grant DMS-1103877.}

\begin{abstract}
In this article we establish the unconditional uniqueness of solutions to an
Infinite Radial Chern-Simons-Schr\"{o}dinger (IRCSS) hierarchy in two
spatial dimensions. The IRCSS hierarchy is a system of infinitely many
coupled PDEs that describes the limiting Chern-Simons-Schr\"{o}dinger
dynamics of infinitely many interacting anyons. The anyons are two
dimensional objects which interact through a self-generated field. Due to
the interactions with the self-generated field, the IRCSS hierarchy is a
system of \emph{nonlinear} PDEs, which distinguishes it from the \emph{linear%
} infinite hierarchies studied previously. Factorized solutions of the IRCSS
hierarchy are determined by solutions of the Chern-Simons-Schr\"{o}dinger
system. Our result therefore implies the unconditional uniqueness of
solutions to the radial Chern-Simons-Schr\"{o}dinger system as well.
\end{abstract}

\maketitle
\tableofcontents


\section{Introduction}

\subsection{The Chern-Simons-Schr\"odinger system}

The Chern-Simons-Schr\"odinger system is given by 
\begin{equation}  \label{CSS-LST}
\begin{cases}
D_t \phi & = i \sum_{\ell = 1}^2 D_\ell D_\ell \phi + i g \lvert \phi
\rvert^2 \phi \\ 
\partial_t A_1 - \partial_1 A_0 & = - \Im(\bar{\phi} D_2 \phi) \\ 
\partial_t A_2 - \partial_2 A_0 & = \Im(\bar{\phi} D_1 \phi) \\ 
\partial_1 A_2 - \partial_2 A_1 & = -\frac{1}{2} \lvert \phi \rvert^2%
\end{cases}%
\end{equation}
where the associated covariant differentiation operators are defined in
terms of the potential $A$ by 
\begin{equation}  \label{D alpha}
D_\alpha := \partial_\alpha+ i A_\alpha, \quad \quad \alpha \in \{0, 1, 2\}
\end{equation}
and where we adopt the convention that $\partial_0 := \partial_t$ and $D_t
:= D_0 $. The wavefunction $\phi$ is complex-valued, the potential $A$ a
real-valued 1-form, and the pair $(A, \phi)$ is defined on $I \times \mathbb{%
R}^2$ for some time interval $I$. The Lagrangian action for this system is 
\begin{equation}  \label{Lagrangian}
L(A,\phi) = \frac12 \int_{\mathbb{R}^{2+1}} \left[ \Im (\bar \phi D_t \phi)
+ |D_x \phi|^2 -\frac{g}2 |\phi|^4 \right] dx dt + \frac12 \int_{\mathbb{R}%
^{2+1}} A \wedge dA
\end{equation}
where here $|D_x \phi|^2 := |D_1 \phi|^2 + |D_2 \phi|^2$. Although the
potential $A$ appears explicitly in the Lagrangian, it is easy to see that
locally $L(A,\phi)$ only depends upon the field $F = dA$. Precisely, the
Lagrangian is invariant with respect to the gauge transformations 
\begin{equation}  \label{gauge-freedom}
\phi \mapsto e^{-i \theta} \phi \quad \quad A \mapsto A + d \theta
\end{equation}
for compactly supported real-valued functions $\theta(t, x)$. The
Chern-Simons-Schr\"odinger system \eqref{CSS-LST}, obtained as the
Euler-Lagrange equations of \eqref{Lagrangian}, inherits this gauge freedom.

The system \eqref{CSS-LST} is a basic model of Chern-Simons dynamics \cite%
{JaPi91, EzHoIw91,EzHoIw91b, JaPi91b}. 
It plays a role in describing certain physical phenomena, such as the
fractional quantum Hall effect, high-temperature superconductivity, and
Aharonov-Bohm scattering, and also provides an example of a
Galilean-invariant planar gauge field theory \cite{JaTe81, DeJaTe82,
JaPiWe90, MaPaSo91, Wi90}.

One interpretation of \eqref{CSS-LST} is as a mean-field equation.
Informally, one may consider \eqref{CSS-LST} as describing the behavior of a
large number of anyons, interacting with each other directly and through a
self-generated field, in the case where the $N$-body wave function
factorizes. There are a number of challenges one encounters in trying to
formalize and prove this statement, and this paper addresses some of them.
We will postpone further discussion of many-body dynamics to the next
subsection, and instead point out that, because the main evolution equation
in \eqref{CSS-LST} includes a cubic nonlinearity, one might hope to prove
for \eqref{CSS-LST} what one can prove for the cubic NLS. It is important to
note, however, that \eqref{CSS-LST} has many nonlinear terms, some nonlocal
and some involving the derivative of the wave-function. These terms appear
because of the geometric structure that arises from modeling the
interactions with the self-generated field. Due to the complexity of the
nonlinearity in \eqref{CSS-LST} and the gauge freedom \eqref{gauge-freedom},
the system \eqref{CSS-LST} is significantly more challenging to analyze than
the cubic NLS. This difference is seen even at the level of the
wellposedness theory, to which we now turn.

The system \eqref{CSS-LST} is Galilean-invariant and has conserved \emph{%
charge} 
\begin{equation}
\chg(\phi ):=\int_{\mathbb{R}^{2}}\lvert \phi \rvert ^{2}dx  \label{charge}
\end{equation}%
and \emph{energy} 
\begin{equation}
E(\phi ):=\frac{1}{2}\int_{\mathbb{R}^{2}}\left[ \lvert D_{x}\phi \rvert
^{2}-\frac{g}{2}\lvert \phi \rvert ^{4}\right] dx  \label{energy}
\end{equation}%
Moreover, for each $\lambda >0$, there is the scaling symmetry 
\begin{equation*}
\begin{split}
& \phi (t,x)\mapsto \lambda \phi (\lambda ^{2}t,\lambda x);\quad
A_{0}(t,x)\mapsto \lambda ^{2}A_{0}(\lambda ^{2}t,\lambda x);\quad
A_{j}(t,x)\mapsto \lambda A_{j}(\lambda ^{2}t,\lambda x),\quad j\in \{1,2\};
\\
& \phi _{0}(x)\mapsto \lambda \phi _{0}(\lambda x),
\end{split}%
\end{equation*}%
which preserves both the system and the charge of the initial data $\phi _{0}
$. Therefore, from the point of view of wellposedess theory, the system %
\eqref{CSS-LST} is $L^{2}$-critical. We remark that system \eqref{CSS-LST}
is defocusing when $g<1$ and focusing when $g\geq 1$. The
defocusing/focusing dichotomy is most readily seen by rewriting the energy %
\eqref{energy} using the so-called Bogomol'nyi identity. After using this
identity, one may also see the dichotomy manifested in the virial and
Morawetz identities. For more details, see \cite[\S \S 4, 5]{LiSm13}. Note
also that the sign convention for $g$ that we adopt, which is the one used
in the Chern-Simons literature, is opposite to the usual one adopted for the
cubic NLS. A more significant difference between Chern-Simons systems and
the cubic NLS is that, unlike the case for the cubic NLS, the coupling
parameter $g$ cannot be rescaled to belong to a discrete set of canonical
values.

Nevertheless, \eqref{CSS-LST} is ill-posed so long as it retains the gauge
freedom \eqref{gauge-freedom}. This freedom is eliminated by imposing an
additional constraint equation. The most common gauge choice for studying %
\eqref{CSS-LST} is the Coulomb gauge, which is the constraint 
\begin{equation}  \label{Coulomb}
\partial_1 A_1 + \partial_2 A_2 = 0
\end{equation}
Coupling \eqref{Coulomb} with the field equations quickly leads to explicit
expressions for $A_\alpha$, $\alpha = 0, 1, 2$, in terms of $\phi$. These
expressions also happen to be nonlinear and nonlocal: 
\begin{equation} \label{CSSConCo}
A_0 = \Delta^{-1} \left[ \partial_1 \Im(\bar{\phi} D_2 \phi) - \partial_2
\Im(\bar{\phi} D_1 \phi) \right], \quad A_1 = \frac12 \Delta^{-1} \partial_2
|\phi|^2, \quad A_2 = -\frac12 \Delta^{-1} \partial_1 |\phi|^2
\end{equation}

Local wellposedness of \eqref{CSS-LST} with respect to the Coulomb gauge at
the Sobolev regularity of $H^2$ is established in \cite{BeBoSa95}. This is
improved to $H^1$ in \cite{Hu13}. Local wellposedness for data small in $H^s$%
, $s > 0$, is established in \cite{LiSmTa13} using the heat gauge, whose
defining condition is $\partial_1 A_1 + \partial_2 A_2 = A_0$. This result
relies upon various Strichartz-type spaces as well as more sophisticated $%
U^p, V^p$ spaces. We refer the reader to \cite[\S 2]{LiSmTa13} for a
comparison of the Coulomb and heat gauges.

In symmetry-reduced settings one may say more, and in particular \cite%
{LiSm13} establishes large data global wellposedness results at the critical
regularity for the equivariant Chern-Simons-Schr\"odinger system. To
introduce the equivariance (or \emph{vortex}) ansatz, it is convenient to
use polar coodinates. Define 
\begin{equation}  \label{A1A2->ArAtheta}
A_r = \frac{x_1}{|x|} A_1 + \frac{x_2}{|x|} A_2, \quad \quad A_\theta = -
x_2 A_1 + x_1 A_2,
\end{equation}
We can invert the transform by writing 
\begin{equation}  \label{ArAtheta->A1A2}
A_1 = A_r \cos \theta - \frac1r A_\theta \sin \theta, \quad \quad A_2 = A_r
\sin \theta + \frac1r A_\theta \cos \theta
\end{equation}
Note that these relations are analogous to 
\begin{equation*}
\partial_r = \frac{x_1}{|x|} \partial_1 + \frac{x_2}{|x|} \partial_2, \quad
\quad \partial_\theta = - x_2 \partial_1 + x_1 \partial_2
\end{equation*}
and 
\begin{equation*}
\partial_1 = (\cos \theta) \partial_r - \frac1r (\sin \theta)
\partial_\theta, \quad \quad \partial_2 = (\sin \theta) \partial_r + \frac1r
(\cos \theta) \partial_\theta
\end{equation*}
The equivariant ansatz, then, is 
\begin{equation}  \label{vortex}
\phi(t, x) = e^{i m \theta} u(t, r), \quad A_1(t, x) = -\frac{x_2}{r} v(t,
r), \quad A_2(t, x) = \frac{x_1}{r} v(t, r), \quad A_0(t, x) = w(t, r)
\end{equation}
where we assume that $m$ is a nonnegative integer, $u$ is real-valued at
time zero, and $v$, $w$ are real-valued for all time. This ansatz implies
that $A_r = 0$ and that $A_\theta$ is a radial function. It also places us
in the Coulomb gauge, i.e., $\partial_1 A_1 + \partial_2 A_2 = 0$, or
equivalently, $\partial_r A_r + \frac1r A_r + \frac{1}{r^2} \partial_\theta
A_\theta = 0$. For some motivation for studying vortex solutions in
Chern-Simons theories, see \cite{PaKh86, VeSc86, VeSc86b, JaWe90, ChSp09,
ByHuSe12}.

Converting \eqref{CSS-LST} into polar coordinates and utilizing %
\eqref{vortex}, we obtain the equivariant Chern-Simons-Schr\"{o}dinger
system (see \cite[\S 1]{LiSm13} for full details):
\begin{equation}
\begin{cases}
(i\partial _{t}+\Delta )\phi  & =\frac{2m}{r^{2}}A_{\theta }\phi +A_{0}\phi +%
\frac{1}{r^{2}}A_{\theta }^{2}\phi -g|\phi |^{2}\phi  \\ 
\partial _{r}A_{0} & =\frac{1}{r}(m+A_{\theta })|\phi |^{2} \\ 
\partial _{t}A_{\theta } & =r\Im (\bar{\phi}\partial _{r}\phi ) \\ 
\partial _{r}A_{\theta } & =-\frac{1}{2}|\phi |^{2}r \\ 
A_{r} & =0%
\end{cases}
\label{equivariant}
\end{equation}%
Global wellposedness holds for equivariant $L^{2}$ data of arbitrary
(nonnegative) charge in the defocusing case $g<1$ and for $L^{2}$ data with
charge less than that of the ground state in the focusing case $g\geq 1$;
this is the main result of \cite{LiSm13}. 

In this paper, we are interested in the radial case ($m = 0$) of system %
\eqref{equivariant}, which is 
\begin{equation}
\begin{cases}
(i\partial _{t}+\Delta )\phi & =A_{0}\phi +\frac{1}{r^{2}}A_{\theta
}^{2}\phi -g|\phi |^{2}\phi \\ 
\partial _{r}A_{0} & =\frac{1}{r}A_{\theta }|\phi |^{2} \\ 
\partial _{t}A_{\theta } & =r\Im (\bar{\phi}\partial _{r}\phi ) \\ 
\partial _{r}A_{\theta } & =-\frac{1}{2}|\phi |^{2}r \\ 
A_{r} & =0%
\end{cases}
\label{eqn:radial}
\end{equation}

\subsection{The infinite Chern-Simons-Schr\"{o}dinger hierarchy}

The infinite Chern-Simons-Schr\"{o}dinger hierarchy is a sequence of trace
class nonnegative operator kernels that are symmetric, in the sense that 
\begin{equation*}
\gamma ^{(k)}(t,\mathbf{x}_{k},\mathbf{x}_{k}^{\prime })=\overline{\gamma
^{(k)}(t,\mathbf{x}_{k}^{\prime },\mathbf{x}_{k})},
\end{equation*}%
and 
\begin{equation}
\gamma ^{(k)}(t,x_{\sigma (1)},\cdots x_{\sigma (k)},x_{\sigma (1)}^{\prime
},\cdots x_{\sigma (k)}^{\prime })=\gamma ^{(k)}(t,x_{1},\cdots
,x_{k},x_{1}^{\prime },\cdots ,x_{k}^{\prime }),  \label{condition:symmetric}
\end{equation}%
for any permutation $\sigma $, and which satisfy the 2D infinite
Chern-Simons-Schr\"{o}dinger hierarchy of equations 
\begin{equation}
\partial _{t}\gamma ^{(k)}+\sum_{j=1}^{k}\left[ iA_{0}(t,x_{j}),\gamma ^{(k)}%
\right] =\sum_{j=1}^{k}\sum_{\ell =1}^{2}i\left[ D_{x_{j}^{(\ell
)}}D_{x_{j}^{(\ell )}},\gamma ^{(k)}\right] +ig\sum_{j=1}^{k}B_{j,k+1}\gamma
^{(k+1)},  \label{eqn:ICSS}
\end{equation}%
where $\R^2 \ni x_j = (x_j^{(1)}, x_j^{(2)})$ for each $j$,
as well as the corresponding field-current identities from \cite[(1.7a)--(1.7c)]{JaPi90}, i.e.,
\begin{equation}
\begin{cases}
F_{01} & =-P_{2}(t,x)-A_{2}(t,x)\rho (t,x) \\ 
F_{02} & =P_{1}(t,x)+A_{1}(t,x)\rho (t,x) \\ 
F_{12} & =-\frac{1}{2}\rho (t,x)%
\end{cases}
\label{eqn:fields}
\end{equation}%
where, as before, $F:=dA$. Here $g$ is the coupling constant, 
\begin{equation}
B_{j,k+1}\gamma ^{(k+1)}:=\Tr_{k+1}\left[ \delta (x_{j}-x_{k+1}),\gamma
^{(k+1)}\right] ,  \label{collision!}
\end{equation}%
the momentum $P(t,x)$ is given by 
\begin{equation*}
P(t,x):=\int e^{i(\xi -\xi ^{\prime })x}\frac{\xi +\xi ^{\prime }}{2}\hat{%
\gamma}^{(1)}(t,\xi ,\xi ^{\prime })d\xi d\xi ^{\prime },
\end{equation*}%
and $\rho (t,x)$ is a shorthand for 
\begin{equation}
\rho (t,x):=\gamma ^{(1)}(t,x,x).  \label{rhodef}
\end{equation}%
Each $x_j \in \mathbb{R}^2$, and $\mathbf{x}_k := (x_1, \dots, x_k) \in 
\mathbb{R}^{2k}$. Given a compactly supported $\theta (t,x)$, the kernels $%
\gamma ^{(k)}$ and potential $A$ transform under a change of gauge according
to 
\begin{equation*}
\gamma ^{(k)}\mapsto \gamma ^{(k)}\prod_{j=1}^{k}e^{-i\theta
(t,x_{j})}e^{i\theta (t,x_{j}^{\prime })},\quad A\mapsto A+d\theta
\end{equation*}%
The invariance of \eqref{eqn:ICSS} and \eqref{eqn:fields} under such
transformations can be checked straightforwardly.

For the purposes of our analysis it is more convenient to write %
\eqref{eqn:ICSS} as 
\begin{equation}  \label{eqn:ICSSrewrite}
\begin{split}
i\partial _{t}\gamma ^{(k)}+\sum_{j=1}^{k}\left[ \Delta _{x_{j}},\gamma
^{(k)}\right] =& \sum_{j=1}^{k}\sum_{\ell =1}^{2}\left[ -2iA_{x_{j}^{(\ell
)}}\partial _{x_{j}^{(\ell )}}-i\partial _{x_{j}^{(\ell )}}A_{x_{j}^{(\ell
)}}+A_{x_{j}^{(\ell )}}^{2},\gamma ^{(k)}\right] \\
&+ \sum_{j=1}^{k} \left[A_0(\cdot, x_j) ,\gamma ^{(k)}\right] - g
\sum_{j=1}^{k}B_{j,k+1} \gamma^{(k+1)}
\end{split}%
\end{equation}
The Coulomb gauge condition \eqref{Coulomb}, upon being
coupled to \eqref{eqn:fields}, leads to 
\begin{equation*}
A_0 = \Delta^{-1} \left[ \partial_1 (P_2 + A_2 \rho) - \partial_2 (P_1 + A_1
\rho) \right], \quad A_1 = \frac12 \Delta^{-1} \partial_2 \rho, \quad A_2 =
-\frac12 \Delta^{-1} \partial_1 \rho
\end{equation*}
This is analogous to how \eqref{CSSConCo} for
the Chern-Simons-Schr\"odinger system \eqref{CSS-LST} is obtained
by coupling to the field equations in \eqref{CSS-LST} the gauge condition \eqref{Coulomb}.
Because each $A_\alpha$ involves $\rho$, defined in \eqref{rhodef}, it is
clear that each term involving $\gamma^{(k)}$ in the right-hand side of %
\eqref{eqn:ICSSrewrite} is best thought of as a nonlinear term. This
nonlinear dependence persists under changes of gauge, though some gauges
lead to tamer nonlinearities than others.

We remark that, while the specific form the nonlinearity of %
\eqref{eqn:ICSSrewrite} takes indeed depends upon the gauge selection made,
the observables associated with the system do not depend upon the gauge
choice.

We note that the system \eqref{CSS-LST} generates a special solution to the
infinite hierarchy \eqref{eqn:ICSS}, \eqref{eqn:fields}. In particular, if $%
(A, \phi)$ solves \eqref{CSS-LST}, then $(A, \{ \gamma^{(k)} \})$ solves %
\eqref{eqn:ICSS}, \eqref{eqn:fields}, where each $\gamma^{(k)}$ is given by 
\begin{equation*}
\gamma ^{(k)}\left( t,\mathbf{x}_{k},\mathbf{x}_{k}^{\prime }\right)
=\prod_{j=1}^{k}\phi (t,x_{j})\bar{\phi}(t,x_{j}^{\prime }).
\end{equation*}

We start our analysis of many-body dynamics with the above infinite
hierarchy. Ideally, one would prefer instead to begin with a many-body
system with only finitely many quantum particles. Because the basic
particles in question are neither bosons nor fermions, there are
difficulties to overcome with such an approach. Concerning the difficulties
in dealing with microscopic statistics, one can refer to \cite{BCEP2}, for
instance. Fortunately, as remarked in \cite{BCEP2}, microscopic statistics
disappear as the particle number tends to infinity. Thus the infinite
hierarchy satisfies the symmetry condition \eqref{condition:symmetric}. We
finally remark that the fact that the field equations \eqref{eqn:fields}
depend merely on the 1-particle density $\gamma ^{(1)}$, as has been
observed formally in the physics literature \cite%
{DeJaTe82,JaPiWe90,JaPi91b,JaTe81,JaWe90}.

One motivation for pursuing an analysis of the infinite hierarchy even
without first specifying the finite hierarchy is that the known approaches
to rigorously deriving mean-field equations, e.g., the Boltzmann equation
and the cubic NLS, all require a uniqueness theorem for the corresponding
infinite hierarchy. Establishing uniqueness of the infinite hierarchy is,
moreover, a critical step. We therefore anticipate that our result in this
article will be the linchpin of any future rigorous derivation of the
Chern-Simons-Schr\"{o}dinger system.

As remarked before, the analysis of the Chern-Simons-Schr\"{o}dinger system
with general data is, at the moment, very delicate. The same remark applies
all the more to the associated infinite hierarchy, to which \eqref{CSS-LST}
is a special solution. Thus we consider the radial version of the infinite
Chern-Simons-Schr\"{o}dinger hierarchy in this paper. The nonradial
equivariant case $(m>0)$, though still much simpler than the general system,
is slightly more challenging than the radial case. Unfortunately, the
techniques we employ for studying the radial case do not immediately extend
to the nonradial equivariant case due to certain logarithmic divergences.

\subsubsection{The Infinite Radial Chern-Simons-Schr\"odinger hierarchy}

The Chern-Simons-Schr\"odinger system \eqref{CSS-LST} simplifies to
\eqref{eqn:radial} under the assumption of radiality. 
Similarly, by assuming radiality, we reduce equations \eqref{eqn:ICSS} through \eqref{rhodef}
to the infinite radial Chern-Simons-Schr\"{o}dinger hierarchy 
\begin{equation}
i\partial _{t}\gamma ^{(k)}+\sum_{j=1}^{k}\left[ \Delta _{x_{j}},\gamma
^{(k)}\right] =\sum_{j=1}^{k}\left[ A_{0}(t,|x_{j}|)+\frac{1}{|x_{j}|^{2}}%
A_{\theta }^{2}(t,|x_{j}|),\gamma ^{(k)}\right] -g\sum_{j=1}^{k}B_{j,k+1}%
\gamma ^{(k+1)}  \label{IR}
\end{equation}%
and the field equations 
\begin{equation*}
F_{r\theta }(t,|x|)=-\frac{1}{2}|x|\rho (t,|x|)
\end{equation*}%
and 
\begin{equation*}
\begin{split}
F_{0\theta }(t,|x|)& =|x|P_{r}(t,|x|) \\
F_{0r}(t,|x|)& =-\frac{1}{|x|}A_{\theta }(t,|x|)\rho (t,|x|),
\end{split}%
\end{equation*}%
for $\gamma ^{(k)}=\gamma ^{(k)}(t,\mathbf{r}_{k},\mathbf{r}_{k}^{\prime })$. 
In particular, here we assume that 
\begin{equation*}
\gamma ^{(k)}=u(t,\mathbf{r}_{k},\mathbf{r}_{k}^{\prime }),\quad
A_{r}=0,\quad A_{\theta }=v(t,r),
\end{equation*}%
where $m$ is a nonnegative integer, $u$ is real-valued at time zero, and $v$
is real-valued for all time. This assumption enforces the Coulomb gauge.
Recall that $B_{j,k+1}$ is defined in \eqref{collision!} and $\rho $ is
given by \eqref{rhodef}. As before, $F:=dA$, though now we are adopting
polar coordinates for $A$. Though we could rewrite everything exclusively in
terms of polar coordinates, we choose instead to use both Cartesian and
polar coordinates.

Putting everything together, we see that we are studying solutions $%
\gamma^{(k)} = \gamma^{(k)}(t, \r_k)$ of 
\begin{equation}  \label{Ihierarchy}
\begin{cases}
i \partial_t \gamma^{(k)} + \sum_{j = 1}^k \left[ \Delta_{x_j}, \gamma^{(k)} %
\right] & = \sum_{j = 1}^k \left[ A_0(t, |x_j|) + \frac{1}{|x_j|^2}
A_\theta^2(t, |x_j|) , \gamma^{(k)} \right] - g \sum_{j = 1}^k B_{j, k+1}
\gamma^{(k+1)} \\ 
\partial_r A_0(t, |x|) & = \frac{1}{|x|} A_\theta \rho(t, |x|) \\ 
\partial_t A_\theta(t, |x|) & = |x| P_r(t, |x|) \\ 
\partial_r A_\theta(t, |x|) & = - \frac12 |x| \rho(t, |x|) \\ 
A_r & = 0%
\end{cases}%
\end{equation}
We interpret $\gamma^{(k)}$ as a complex-valued function on $\mathbb{R}_t
\times \mathbb{R}^{k}_+ \times \mathbb{R}^{k}_+$ subject to the symmetries 
\begin{equation*}
\gamma^{(k)}(t, \r_k, \r_k^\prime) = \overline{\gamma^{(k)}(t, \r_k^\prime,
\r_k)}
\end{equation*}
and 
\begin{equation}  \label{indistinguishability}
\gamma^{(k)}(t, r_{\sigma(1)}, \ldots, r_{\sigma(k)}, r_{\sigma(1)}^\prime,
\ldots, r_{\sigma(k)}^\prime) = \gamma^{(k)}(t, r_1, \ldots, r_k,
r_1^\prime, \ldots, r_k^\prime)
\end{equation}
Though each $r_j \in \mathbb{R}_+$, we associate to this space the measure $%
r dr$, as indeed we think of $r_j = |x_j|$ for $x_j \in \mathbb{R}^2$.

Note that we can eliminate $A_\theta, A_0$ in \eqref{Ihierarchy}. In
particular, we have 
\begin{equation}  \label{AthetaDef}
A_\theta(t, r) = - \frac12 \int_0^{r} \rho(t, s) s ds
\end{equation}
and 
\begin{equation}  \label{A0Def}
A_0(t, r) = \frac12 \int_{r}^\infty \rho(t, s) \int_0^s \rho(t, u) u du 
\frac{ds}{s}
\end{equation}
which reflect the natural boundary conditions for $A_\theta$, $A_0$ that we
adopt for \eqref{CSS-LST}. Therefore we may rewrite \eqref{Ihierarchy} as 
\begin{equation}  \label{Ih}
\begin{split}
i \partial_t \gamma^{(k)} + \sum_{j = 1}^k [\Delta_{x_j}, \gamma^{(k)}] =&
\sum_{j = 1}^k \left[ \frac12 \int_{r_j}^\infty \rho(t, s) \int_0^s \rho(t,
u) u du \frac{ds}{s} + \frac{1}{r_j^2} \left( - \frac12 \int_0^{r_j} \rho(t,
s) s ds \right)^2, \gamma^{(k)} \right] \\
& - g \sum_{j = 1}^k B_{j, k+1} \gamma^{(k+1)} \\
\gamma^{(k)}(0) =& \; \gamma_0^{(k)}, \quad k \in \mathbb{N}
\end{split}%
\end{equation}

\subsection{Main results}

Our main theorem says that any admissible mild solution of the radial
infinite CSS hierarchy is unconditionally unique in $L_{t\in \lbrack
0,T)}^{\infty }\mathfrak{H}_{\mathrm{rad}}^{\frac{2}{3}}$. To explain what
this means, for $s\in \mathbb{R}$, we define the space $\mathfrak{H}_{%
\mathrm{rad}}^{s}$ to be the collection of sequences $\{\gamma
^{(k)}\}_{k\in \mathbb{N}}$ of density matrices in $L_{\mathrm{sym}}^{2}(%
\mathbb{R}^{2k})$ such that $\gamma ^{(k)}=\gamma ^{(k)}(t,\mathbf{r}_{k},%
\mathbf{r}_{k}^{\prime })$ and 
\begin{equation*}
\Tr(|S^{(k,s)}\gamma ^{(k)}|)<M^{2k},\quad \text{for all }k\in \mathbb{N}%
\text{ and for some constant }M>0
\end{equation*}%
where 
\begin{equation*}
S^{(k,s)}:=\prod_{j=1}^{k}(1-\Delta _{x_{j}})^{\frac{s}{2}}(1-\Delta
_{x_{j}^{\prime }})^{\frac{s}{2}}
\end{equation*}%
Here $L_{\mathrm{sym}}^{2}$ denotes the space of $L^{2}$ functions
satisfying \eqref{condition:symmetric}. Let $U^{(k)}(t)$ denote the
propagator 
\begin{equation}
U^{(k)}(t):=e^{it\Delta _{\mathbf{x}_{k}}}e^{-it\Delta _{\mathbf{x}%
_{k}^{\prime }}}  \label{bigUprop}
\end{equation}

A \emph{mild solution} of \eqref{Ih} in the space $L_{[0,T]}^{\infty }%
\mathfrak{H}_{\mathrm{rad}}^{s}$ is a sequence of marginal density matrices $%
\Gamma =(\gamma ^{(k)}(t))_{k\in \mathbb{N}}$ solving 
\begin{equation*}
\begin{split}
\gamma ^{(k)}(t)=& \;U^{(k)}(t)\gamma
^{(k)}(0)-i\int_{0}^{t}U^{(k)}(t-s)\times \\
& \left( \sum_{j=1}^{k}\left[ \frac{1}{2}\int_{r_{j}}^{\infty }\rho
(t,v)\int_{0}^{v}\rho (t,u)udu\frac{dv}{v}+\frac{1}{r_{j}^{2}}\left( -\frac{1%
}{2}\int_{0}^{r_{j}}\rho (t,v)vdv\right) ^{2},\gamma ^{(k)}\right] \right. \\
& \left. \quad -g\sum_{j=1}^{k}B_{j,k+1}\gamma ^{(k+1)}\right) ds
\end{split}%
\end{equation*}%
and satisfying 
\begin{equation*}
\sup_{t\in \lbrack 0,T)}\Tr(|S^{(k,s)}\gamma ^{(k)}(t)|)<M^{2k} 
\end{equation*}
for a finite constant $M$ independent of $k$. Note that if we are given
factorized initial data 
\begin{equation*}
\gamma _{0}^{(k)}(\r_k, \r_{k}^{\prime})=\prod_{j=1}^{k}\phi _{0}(r_{j})%
\overline{\phi _{0}(r_{j}^{\prime })},
\end{equation*}%
then the condition that $(\gamma ^{(k)}(0))\in \mathfrak{H}_{\mathrm{rad}%
}^{s}$ is equivalent to 
\begin{equation*}
\Tr(|S^{(k,s)}\gamma ^{(k)}(0)|)=\Vert \phi _{0}\Vert
_{H^{s}}^{2k}<M^{2k},\quad k\in \mathbb{N},
\end{equation*}%
which is to say that $\Vert \phi _{0}\Vert _{H^{s}}<M$ for some $M<\infty $.
Then a solution to the IRCSS hierarchy in $L_{t\in \lbrack 0,T)}^{\infty }%
\mathfrak{H}_{\mathrm{rad}}^{s}$ is given by the sequence of factorized
density matrices 
\begin{equation*}
\gamma ^{(k)}(t, \r_k, \r_k^\prime)=\prod_{j=1}^{k}\phi _{t}(r_{j})\overline{%
\phi _{t}(r_{j}^{\prime })}
\end{equation*}%
provided the corresponding 1-particle wave function $\phi _{t}$ satisfies
the radial Chern-Simons-Schr\"{o}dinger system \eqref{eqn:radial}.

Admissibility we take to mean that $\Tr \gamma^{(k)} = 1$ for all $k \in \N$ and
\begin{equation}
\gamma^{(k)}=\Tr_{k+1}(\gamma ^{(k+1)}),\quad k\in \mathbb{N}.
\label{admissible}
\end{equation}%
This is required in our application of the quantum de Finetti theorem. As
there are weak analogues of the quantum de Finetti theorem applicable to
limiting hierarchies, we expect our techniques to apply to the problem of
rigorously deriving the radial CSS from large, finite systems.

\begin{thm}[Unconditional uniqueness for the infinite hierarchy]
\label{thm:main}There is at most one $L_{t\in \lbrack 0,T)}^{\infty }%
\mathfrak{H}_{\mathrm{rad}}^{\frac23}$ admissible solution to the infinite
radial Chern-Simons-Schr\"{o}dinger hierarchy \eqref{Ihierarchy}.
\end{thm}

\begin{thm}[Unconditional uniqueness for the Chern-Simons-Schr\"{o}dinger
system]
\label{thm:2} There is at most one $L_{t\in \lbrack 0,T)}^{\infty }H^{\frac{2%
}{3}}(\mathbb{R}^{2})$ solution to the radial Chern-Simons-Schrodinger
system \eqref{eqn:radial}.
\end{thm}

Before explaining our main theorem, we first remark that deriving mean-field
equations from many-body systems by studying infinite hierarchies is a very
rich subject. For works related to the Boltzmann equation, see \cite%
{Lanford,King,ArCaIa91,DiluteGasBook,Saint-Raymond}. For works related to
the Hartree equation, see \cite%
{Spohn,Frolich,E-Y1,RodnianskiAndSchlein,KnowlesAndPickl,GMM1,GMM2,Chen2ndOrder,ChLeSc11,MichelangeliSchlein,AmNi08,AmNi11,LeNaRo14}%
. For works related to the cubic NLS, see \cite%
{AGT,E-E-S-Y1,E-S-Y1,E-S-Y2,E-S-Y3,E-S-Y5,KlMa08,KiScSt11,TChenAndNP,TChenAndNpGP1,TCNPNT1,TChenAndNPSpace-Time,Pickl,ChenAnisotropic,Chen3DDerivation,BeOlSc12,GM1,C-H3Dto2D,C-H2/3,ChHaPaSe13,C-HFocusing,HoTaXi14,GrSoSt14,SoSt13,Sohinger2,Sohinger3}%
. For works related to the quantum Boltzmann equation, see \cite%
{BCEP1,BCEP2,BCEP4,ErdosFock}. The infinite hierarchies considered
previously to the present one are all linear. In contrast to this, the
infinite radial Chern-Simons-Schr\"{o}dinger hierarchy is nonlinear.

For our problem we have taken the phrase \textquotedblleft unconditional
uniqueness" from the study of the NLS. It is shown by Cercignani's
counterexample \cite{DiluteGasBook} that solutions to infinite hierarchies
like the Boltzmann hierarchy and the Gross-Pitaevskii hierarchy are
generally not unconditionally unique in the sense that a solution is not
uniquely determined by the initial datum unless one assumes appropriate
space-time bounds on the solution. In the NLS literature, \textquotedblleft
unconditional uniqueness" usually means establishing uniqueness without
assuming that some Strichartz norm is finite. Since we are using tools from
the study of the NLS, we therefore call our main theorems unconditional
uniqueness theorems.\footnote{%
In other words, the uniqueness theorems regarding the Gross-Pitaevskii
hierarchies \cite{KlMa08,KiScSt11,ChenAnisotropic,C-HFocusing,GrSoSt14} are
conditional, whereas \cite{AGT,E-S-Y2,ChHaPaSe13,HoTaXi14,Sohinger3} are
unconditional in the NLS sense. Yet, they are all considered conditional in
the Boltzmann literature.}

Finally, we remark that, for the proof of the main theorems, we apply the
quantum de Finetti theorem similarly to as is done in \cite%
{ChHaPaSe13,HoTaXi14}, but with adjustments tailored to deal with the
nonlinearity in the infinite hierarchy that we consider. The quantum de
Finetti theorem is a version of the classical Hewitt-Savage theorem. T.
Chen, C. Hainzl, N. Pavlovic, and R. Seiringer are the first to apply the
quantum de Finetti theorem to the study of infinite hierarchies in the
quantum setting. For results regarding the uniqueness of the Boltzmann
hierarchy using the Hewitt-Savage theorem, see \cite{ArCaIa91}.

\section{Proof of the Main Theorem}

We will prove that if we are given two $L_{[0,T]}^{\infty }\mathfrak{H}_{%
\mathrm{rad}}^{\frac{2}{3}}$ solutions $\{\gamma _{1}^{(k)}\}$ and $\{\gamma
_{2}^{(k)}\}$ to system \eqref{Ihierarchy} subject to the same initial
datum, then the trace norm of the difference $\{\gamma ^{(k)}=\gamma
_{1}^{(k)}-\gamma _{2}^{(k)}\}$ is zero. In contrast to the usual infinite
hierarchies, (e.g. Boltzmann, Gross-Pitaevskii, ...), system %
\eqref{Ihierarchy} is nonlinear. Thus $\gamma ^{(k)}$ does not solve system %
\eqref{Ihierarchy}. In order to show that $\gamma ^{(k)}$ has zero trace
norm, we first express $\gamma ^{(k)}$ as a suitable Duhamel-Born series,
which contains a nonlinear part and an interaction part (see \S \ref%
{sec:setup for the proof}). These two parts we estimate separately, with
bounds contained respectively in Theorems \ref{Thm:estimate for NP} and \ref%
{Thm:estimate for IP}, which together constitute our main estimates. In \S %
\ref{Sec:ProofWithMainEstimates} we prove the main theorem, Theorem \ref%
{thm:main}, assuming the main estimates. The proof of Theorem \ref%
{Thm:estimate for NP} is postponed to \S \ref{sec:multi} (and Theorem \ref%
{Thm:estimate for IP} we handle in this section).

\subsection{Setup\label{sec:setup for the proof}}

Set for short 
\begin{equation}
a\left( r_{j}\right) := A_{0}(t,r_{j})+\frac{1}{r_{j}^{2}}A_{\theta
}^{2}(t,r_{j})  \label{a-def}
\end{equation}%
and 
\begin{equation}
a\left( \mathbf{r}_{k}\right) := \sum_{j=1}^{k}a(r_{j})  \label{a-sum}
\end{equation}%
Let $\mathcal{A}^{(k)}$ denote the operator that acts according to%
\begin{equation}
\mathcal{A}^{(k)}f := \left[ a(\mathbf{r}_{k}),f\right]  \label{cA-def}
\end{equation}%
Also, set for short 
\begin{equation}
B_{k+1}:=\sum_{j=1}^{k}B_{j,k+1}=\sum_{j=1}^{k}\Tr\nolimits_{k+1}\left[
\delta (x_{j}-x_{k+1}),\cdot \right]  \label{Bshort}
\end{equation}
With these abbreviations, the first equation of \eqref{Ihierarchy} assumes
the form 
\begin{equation}
i\partial _{t}\gamma^{(k)}+\left[ \Delta _{\mathbf{x}_{k}},\gamma^{(k)}%
\right] = \mathcal{A}^{(k)}\gamma^{(k)}-gB_{k+1}\gamma^{(k+1)}  \label{GP}
\end{equation}

\begin{rem}
The operator $\mathcal{A}^{(k)}$ is linear, but itself depends upon $\gamma
^{(1)}$. In fact, it only depends upon the diagonal $\rho (t,r)=\gamma
^{(1)}(t,r,r)$. The term $\mathcal{A}^{(k)}\gamma ^{(k)}$ is therefore
better thought of as a nonlinear term rather than a linear one.
\end{rem}

Let $\{\gamma _{1}^{(k)}\}$ and $\{\gamma _{2}^{(k)}\}$ be solutions subject
to the same initial data, with, respectively, $\rho _{1}(t,r):=\gamma
_{1}^{(1)}(t,r,r)$ and $\rho _{2}(t,r):=\gamma _{2}^{(1)}(t,r,r)$. Let $%
\gamma ^{(k)}:=\gamma _{1}^{(k)}-\gamma _{2}^{(k)}$. Then 
\begin{equation}
i\partial _{t}\gamma ^{(k)}+\left[ \Delta _{\mathbf{x}_{k}},\gamma ^{(k)}%
\right] =\mathcal{A}_{1}^{(k)}\gamma _{1}^{(k)}-\mathcal{A}_{2}^{(k)}\gamma
_{2}^{(k)}-gB_{k+1}\gamma ^{(k+1)}  \label{I-diff1}
\end{equation}%
We can rewrite \eqref{I-diff1} using the relation 
\begin{equation*}
\mathcal{A}_{1}^{(k)}\gamma _{1}^{(k)}-\mathcal{A}_{2}^{(k)}\gamma
_{2}^{(k)}=\mathcal{A}_{1}^{(k)}\gamma ^{(k)}+\mathcal{A}^{(k)}\gamma
_{2}^{(k)},
\end{equation*}%
where now 
\begin{equation*}
\mathcal{A}^{(k)}:=\mathcal{A}_{1}^{(k)}-\mathcal{A}_{2}^{(k)},
\end{equation*}%
so that it becomes 
\begin{equation}
i\partial _{t}\gamma ^{(k)}+\left[ \Delta _{\mathbf{x}_{k}},\gamma ^{(k)}%
\right] =\mathcal{A}_{1}^{(k)}\gamma ^{(k)}+\mathcal{A}^{(k)}\gamma
_{2}^{(k)}-gB_{k+1}\gamma ^{(k+1)},  \label{I-diff2}
\end{equation}%
or, equivalently, 
\begin{equation*}
(i\partial _{t}+\Delta _{\mathbf{x}_{k}}-\Delta _{\mathbf{x}_{k}^{\prime
}})\gamma ^{(k)}=\mathcal{A}_{1}^{(k)}\gamma ^{(k)}+\mathcal{A}^{(k)}\gamma
_{2}^{(k)}-gB_{k+1}\gamma ^{(k+1)}.
\end{equation*}%
Recalling the corresponding linear propagator $U^{(k)}(t)$ defined in %
\eqref{bigUprop}, we write \eqref{I-diff2} in integral form, i.e.,%
\begin{equation}
\gamma ^{(k)}(t_{k})= - i g \int_{0}^{t_{k}}dt_{k+1}U^{(k)}(t_{k}-t_{k+1})\left[ 
\mathcal{A}_{1}^{(k)}\gamma ^{(k)}(t_{k+1})+\mathcal{A}^{(k)}\gamma
_{2}^{(k)}(t_{k+1})+B_{k+1}\gamma ^{(k+1)}(t_{k+1})\right] ,
\label{eqn:difference in duhamel form}
\end{equation}%
In invoking this formula in future calculations, we set $g = -1$ for simplicity and
we ignore the $i$ in front so that we
do not need to keep track of its exact power, as the precise power is not
relevant to the estimates.

\begin{rem}
The choice of $g=-1$ corresponds to a defocusing case in \eqref{equivariant}%
. It is important to note, however, that the choice $g=-1$ at this step is
purely for the sake of convenience; all subsequent arguments can accommodate
any $g\neq -1$ at the cost of certain powers of $\left\vert g\right\vert $.
In particular, our arguments apply to the self-dual case $g=1$, which is the
most interesting from the physical point of view.
\end{rem}

For the purpose of proving unconditional uniqueness, it suffices to show $%
\gamma ^{(1)}=0$. Iterating \eqref{eqn:difference in duhamel form} $l_{c}$
times\footnote{%
Here, $l_{c}$ stands for the level of coupling. When $l_{c}=0$, one recovers %
\eqref{I-diff2}.}, we obtain%
\begin{equation}  \label{eqn:duhamel before board game}
\begin{split}
\gamma^{(1)}(t_{1}) =&\int_{0}^{t_{1}}dt_{2}U^{(1)}(t_{1}-t_{2})\left( 
\mathcal{A}_{1}^{(1)}\gamma ^{(1)}(t_{2})+\mathcal{A}^{(1)}\gamma
_{2}^{(1)}(t_{2})\right)
+\int_{0}^{t_{1}}dt_{2}U^{(1)}(t_{1}-t_{2})B_{2}\gamma ^{(2)}(t_{2}) \\
=&\int_{0}^{t_{1}}dt_{2}U^{(1)}(t_{1}-t_{2})\left( \mathcal{A}%
_{1}^{(1)}\gamma ^{(1)}(t_{2})+\mathcal{A}^{(1)}\gamma
_{2}^{(1)}(t_{2})\right) \\
&+\int_{0}^{t_{1}}dt_{2}U^{(1)}(t_{1}-t_{2})B_{2}%
\int_{0}^{t_{2}}dt_{3}U^{(2)}(t_{2}-t_{3})\left( \mathcal{A}_{1}^{(2)}\gamma
^{(2)}(t_{3})+\mathcal{A}^{(2)}\gamma _{2}^{(2)}(t_{3})\right) \\
&+\int_{0}^{t_{1}}dt_{2}U^{(1)}(t_{1}-t_{2})B_{2}%
\int_{0}^{t_{2}}dt_{3}U^{(2)}(t_{2}-t_{3})B_{3}\gamma ^{(3)}(t_{3}) \\
=& \; \ldots \\
=& \; \mathrm{NP}^{(l_{c})}+ \mathrm{IP}^{(l_{c})}
\end{split}%
\end{equation}%
where $\mathrm{NP}^{(l_{c})}$ and $\mathrm{IP}^{(l_{c})}$, the nonlinear
part and the interaction part, respectively, are given by 
\begin{equation}
\mathrm{NP}^{(l_{c})}=G^{(1)}+\sum_{r=1}^{l_{c}}\int_{0}^{t_{1}} \cdots
\int_{0}^{t_{r}}dt_{2}\cdots dt_{r+1}U^{(1)}(t_{1}-t_{2})B_{2}\cdots
U^{(r)}(t_{r}-t_{r+1})B_{r+1}G^{(r+1)}(t_{r+1})  \label{eqn:NP}
\end{equation}%
and 
\begin{equation}
\mathrm{IP}^{(l_{c})}=\int_{0}^{t_{1}} \cdots
\int_{0}^{t_{l_{c}+1}}dt_{2}\cdots
dt_{l_{c}+1}U^{(1)}(t_{1}-t_{2})B_{2}...U^{(l_{c}+1)}(t_{l_{c}}-t_{l_{c}+1})B_{l_{c}+2}\gamma ^{(l_{c}+2)}(t_{l_{c}+2})
\label{eqn:IP}
\end{equation}
where 
\begin{equation*}
G^{(k)}(t_{k}) := \int_{0}^{t_{k}}dt_{k+1}U^{(k)}(t_{k}-t_{k+1})\left( 
\mathcal{A}_{1}^{(k)}\gamma ^{(k)}(t_{k+1})+\mathcal{A}^{(k)}\gamma
_{2}^{(k)}(t_{k+1})\right)
\end{equation*}

\subsection{Proof Assuming the Main Estimates\label%
{Sec:ProofWithMainEstimates}}

\begin{thm}
\label{Thm:estimate for NP}There exists a constant $C>0$ such that%
\begin{equation*}
\Tr\left\vert \mathrm{NP}^{(l_{c})}\left( t_{1}\right) \right\vert \leq
Ct_{1}\sup_{t\in \left[ 0,t_{1}\right] }\Tr\left\vert \gamma
^{(1)}(t)\right\vert
\end{equation*}%
for all coupling levels $l_{c}$ and all sufficiently small $t_{1}$.
\end{thm}

\begin{proof}
We postpone the proof to \S \ref{Sec:main nl estimate}.
\end{proof}

\begin{thm}
\label{Thm:estimate for IP}There exists a constant $C>0$ such that%
\begin{equation*}
\Tr\left\vert \mathrm{IP}^{(l_{c})}\left( t_{1}\right) \right\vert \leq
\left( Ct_{1}^{\frac{1}{3}}\right) ^{l_{c}}
\end{equation*}%
for all coupling levels $l_{c}.$
\end{thm}

\begin{proof}
This estimate follows from the same method used for the corresponding term
in \cite{ChHaPaSe13}, which relies on the Quantum de Finetti theorem and on
a combinatorial analysis of the graphs that one can associate to the Duhamel expansions. 
One merely needs to replace the 3D trilinear estimates 
\cite[(6.19) and (6.20)]{ChHaPaSe13} with \eqref{TH} and \eqref{TL2},
respectively, taking $s=\frac{2}{3}$, and replace the 3D Sobolev estimate%
\begin{equation*}
\left\Vert f\right\Vert _{L^{6}(\mathbb{R}^{3})}\lesssim \left\Vert f\right\Vert
_{H^{1}(\mathbb{R}^{3})}
\end{equation*}%
with the 2D Sobolev estimate%
\begin{equation*}
\left\Vert f\right\Vert _{L^{6}(\mathbb{R}^{2})}\lesssim \left\Vert f\right\Vert
_{H^{\frac{2}{3}}(\mathbb{R}^{2})}.
\end{equation*}
We remark that it is because of this Sobolev estimate that we take $s = 2/3$
in $H^{s}$ rather than a smaller $s$.
\end{proof}

With Theorems \ref{Thm:estimate for NP} and \ref{Thm:estimate for IP}, we
then infer from \eqref{eqn:duhamel before board game} that%
\begin{align*}
\Tr\left\vert \gamma ^{(1)}(t_{1})\right\vert &\leq \Tr\left\vert \mathrm{NP}%
^{(l_{c})}\left( t_{1}\right) \right\vert +\Tr\left\vert \mathrm{IP}%
^{(l_{c})}\left( t_{1}\right) \right\vert \\
&\leq Ct_{1}\sup_{t\in \left[ 0,t_{1}\right] }\Tr\left\vert \gamma
^{(1)}(t)\right\vert +\left( Ct_{1}^{\frac{1}{3}}\right) ^{l_{c}} \\
&\leq CT\sup_{t\in \left[ 0,T\right] }\Tr\left\vert \gamma
^{(1)}(t)\right\vert +\left( CT^{\frac{1}{3}}\right) ^{l_{c}}
\end{align*}%
for all $t_{1}\in \lbrack 0,T].$ Take the supremum in time on both sides to
get 
\begin{equation*}
\sup_{t\in \left[ 0,T\right] }\Tr\left\vert \gamma ^{(1)}(t)\right\vert \leq
CT\sup_{t\in \left[ 0,T\right] }\Tr\left\vert \gamma ^{(1)}(t)\right\vert
+\left( CT^{\frac{1}{3}}\right) ^{l_{c}}.
\end{equation*}%
Therefore, for all $T$ small enough, we obtain%
\begin{equation*}
\frac{1}{2}\sup_{t\in \left[ 0,T\right] }\Tr\left\vert \gamma
^{(1)}(t)\right\vert \leq \left( CT^{\frac{1}{3}}\right) ^{l_{c}}\rightarrow
0\text{ as }l_{c}\rightarrow \infty,
\end{equation*}%
i.e.,%
\begin{equation*}
\sup_{t\in \left[ 0,T\right] }\Tr\left\vert \gamma ^{(1)}(t)\right\vert =0.
\end{equation*}
Hence we have finished the proof of the main theorem assuming Theorem \ref%
{Thm:estimate for NP}. The bulk of the rest of the paper is devoted to
proving Theorem \ref{Thm:estimate for NP}.


\section{Estimate for the Nonlinear Part\label{Sec:main nl estimate}}

Recall 
\begin{align*}
\mathrm{NP}^{(l_{c})} &=G^{(1)}+\sum_{r=1}^{l_{c}}\int_{0}^{t_{1}} \cdots
\int_{0}^{t_{r}}dt_{2} \cdots dt_{r+1}U^{(1)}(t_{1}-t_{2})B_{2}\cdots
U^{(r)}(t_{r}-t_{r+1})B_{r+1}G^{(r+1)}(t_{r+1}) \\
&=: \mathrm{I} + \mathrm{II},
\end{align*}
where 
\begin{equation}  \label{Gk}
G^{(k)}(t_{k})=\int_{0}^{t_{k}}dt_{k+1}U^{(k)}(t_{k}-t_{k+1}) \left( 
\mathcal{A}_{1}^{(k)}\gamma ^{(k)}(t_{k+1})+\mathcal{A}^{(k)}
\gamma_{2}^{(k)}(t_{k+1})\right) .
\end{equation}
We will first treat $\Tr\left\vert G^{(1)}(t_{1})\right\vert $ coming from
part $\mathrm{I}$ and then, with some additional tools, the corresponding
term coming from part $\mathrm{II}$. Both of the estimates rely upon the
quantum de Finetti theorem stated below.

\begin{thm}[Quantum de Finetti theorem \protect\cite%
{HudsonMoody,Stormer-69,AmNi08,AmNi11,LeNaRo14}]
\label{Thm:de finetti} Let ${\mathcal{H}}$ be a separable Hilbert space and
let ${\mathcal{H}}^{k}=\bigotimes_{sym}^{k}{\mathcal{H}}$ denote the
corresponding bosonic $k$-particle space. Let $\Gamma $ denote a collection
of bosonic density matrices on ${\mathcal{H}}$, i.e., 
\begin{equation*}
\Gamma =(\gamma ^{(1)},\gamma ^{(2)},\dots )
\end{equation*}%
with $\gamma ^{(k)}$ a non-negative trace class operator on ${\mathcal{H}}%
^{k}$. If $\Gamma $ is admissible, i.e., 
for all $k \in \N$ we have
$\Tr \gamma^{(k)} = 1$ and
$\gamma ^{(k)}=\mathrm{Tr}%
_{k+1}\gamma ^{(k+1)}$, where $\mathrm{Tr}_{k+1}$ denotes the partial trace
over the $(k+1)$-th factor, then there exists a
unique Borel probability measure $\mu $, supported on the unit sphere in ${%
\mathcal{H}}$, and invariant under multiplication of $\phi \in {\mathcal{H}}$
by complex numbers of modulus one, such that 
\begin{equation*}
\gamma ^{(k)}=\int d\mu (\phi )(|\phi \rangle \langle \phi |)^{\otimes
k},\quad \forall k\in {\mathbb{N}}\,.
\end{equation*}
\end{thm}

\begin{rem}
The $\mu $ determined by Theorem \ref{Thm:de finetti} is finite and so, in
particular, $\sigma$-finite. Therefore, the Fubini-Tonelli theorem, which is
crucial in the proof, applies. See \cite[p.~190]{DuSc88}.
\end{rem}

Using Theorem \ref{Thm:de finetti}, we write 
\begin{equation*}
\gamma _{j}^{(k)}(t)=\int d\mu _{t}^{(j)}(\phi )(|\phi \rangle \langle \phi
|)^{\otimes k},\quad j=1,2,
\end{equation*}%
and 
\begin{equation*}
\gamma ^{(k)}(t)=\int d\mu _{t}(\phi )(|\phi \rangle \langle \phi
|)^{\otimes k}
\end{equation*}%
where $\mu _{t}:=\mu _{t}^{(1)}-\mu _{t}^{(2)}$ is a signed measure
supported on the unit sphere of $L^{2}(\mathbb{R}^2)$. We remark that%
\begin{equation*}
\Tr\left\vert \gamma ^{(1)}(t)\right\vert =\int d\left\vert \mu
_{t}\right\vert (\phi )\left\Vert \phi \right\Vert _{L^{2}}^{2}=\int
d\left\vert \mu _{t}\right\vert (\phi )
\end{equation*}%
while%
\begin{equation*}
\Tr\left\vert \gamma _{j}^{(1)}(t)\right\vert =\int d\mu
_{t}^{(j)}\left\Vert \phi \right\Vert _{L^{2}}^{2}=\int d\mu _{t}^{(j)}=1.
\end{equation*}%
Here $\left\vert \mu _{t}\right\vert $ is defined, in the usual way, as the
sum of the positive part and the negative part of $\mu _{t}$, which itself
is another finite measure since $\left\vert \mu _{t}\right\vert \leq \mu
_{t}^{(1)}+\mu _{t}^{(2)}$. Write $\mu _{t}^{(0)}=\mu _{t}$ for convenience.
The main properties of $\mu _{t}^{(i)}$ that we need are%
\begin{equation}
\sup_{t\in \left[ 0,T\right] }\int d\left\vert \mu _{t}^{(i)}\right\vert
(\phi )\Vert \phi \Vert _{H_{x}^{\frac23}}^{2k}\leq M^{2k} \quad \text{for}
\quad i=0, 1, 2,  \label{measure energy:integral}
\end{equation}
and 
\begin{equation}  \label{measure energy:support}
\left\vert \mu _{t}^{(i)}\right\vert \left( \left\{ \phi \in L^{2}(\mathbb{R}%
^{2})|\left\Vert \phi \right\Vert _{H^{\frac{2}{3}}}>M\right\} \right)
=0\quad \text{for} \quad i=0,1,2,
\end{equation}
where $\left\vert \mu _{t}^{(i)}\right\vert $ is of course $\mu _{t}^{(i)}$
if $i=1$ or $2$. For $i=1$, $2$, estimate \eqref{measure energy:integral} is
equivalent to the energy condition%
\begin{equation}
\sup_{t\in \left[ 0,T\right] }\Tr\left( \prod_{j=1}^{k}\left\langle \nabla
_{x_{j}}\right\rangle ^{\frac{2}{3}}\right) \gamma _{i}^{(k)}(t)\left(
\prod_{j=1}^{k}\left\langle \nabla _{x_{j}^{\prime }}\right\rangle ^{\frac{2%
}{3}}\right) \leq M^{2k}\text{ for }i=1,2,  \label{energy condition}
\end{equation}%
and \eqref{measure energy:support} then follows from 
\eqref{measure
energy:integral} using Chebyshev's inequality.\footnote{%
See \cite[Lemma 4.4]{ChHaPaSe13} or \cite[(2.17)]{HoTaXi14}.} The $i=0$ case
then follows from the definition.

Putting these structures into $\mathcal{A}$, for $\ell =1,2$, we have%
\begin{equation} \label{newlab1}
\mathcal{A}_{\ell }^{(k)}f(t)=\iint d\mu _{t}^{(\ell )}(\psi )d\mu
_{t}^{(\ell )}(\omega ) \sum_{j=1}^{k}\left[ a_{|\psi|^2,
|\omega|^2}(r_{j})- a_{|\psi|^2, |\omega|^2}(r_{j}^{\prime })\right] f
\end{equation}%
and%
\begin{equation} \label{newlab2}
\begin{split}
\mathcal{A}^{(k)}f(t) =&\; \left( \mathcal{A}_{1}^{(k)}-\mathcal{A}%
_{2}^{(k)}\right) f \\
=&\; \iint d\mu _{t}^{(1)}(\psi )d\mu _{t}(\omega )\sum_{j=1}^{k}\left[
a_{|\psi|^2, |\omega|^2}(r_{j})-a_{|\psi|^2, |\omega|^2}(r_{j}^{\prime })%
\right] f \\
&+\iint d\mu _{t}(\psi )d\mu _{t}^{(2)}(\omega )\sum_{j=1}^{k}\left[
a_{|\psi|^2, |\omega|^2}(r_{j})-a_{|\psi|^2, |\omega|^2}(r_{j}^{\prime })%
\right] f
\end{split}
\end{equation}
where $a_{|\psi|^2, |\omega|^2}$ is defined by 
\begin{equation*}  
a_{|\psi|^2, |\omega|^2}(t, r) := A_0^{(|\psi|^2, |\omega|^2)}(t, r) + \frac{1}{r^2}
A_\theta^{(|\psi|^2)}(t, r) A_\theta^{(|\omega|^2)}(t, r)
\end{equation*}
with
\[
A_0^{(|\psi|^2, |\omega|^2)}(t, r) = - \int_r^\infty A_\theta^{(|\psi|^2)}(t, s) |\omega|^2(t, s) \frac{ds}{s}, \quad
A_\theta^{(\rho)}(t, r) = - \frac12 \int_0^r \rho(t, s) s ds
\]
Informally speaking, $a_{|\psi|^2, |\omega|^2}(r)$ is similar to $a(r)$
defined in \eqref{a-def}, but is linear with respect to $|\psi|^2$ and $%
|\omega|^2$ independently rather than quadratic with respect to a single $%
|\phi|^2$. 

This notation enables us to represent the core term of $G^{(k)}$ by 
\begin{equation}
\begin{split}
\mathcal{A}_{1}^{(k)}& \gamma ^{(k)}(t)+\mathcal{A}^{(k)}\gamma _{2}^{(k)}(t)
\\
& =\sum_{j=1}^{k}\sum_{\left( l,m,n\right) \in \mathcal{P}}\iiint d\mu
_{t}^{(l)}(\psi )d\mu _{t}^{(m)}(\omega )d\mu _{t}^{(n)}(\phi )\left[
a_{|\psi|^2, |\omega|^2}(r_{j})-a_{|\psi|^2, |\omega|^2}(r_{j}^{\prime })%
\right] (|\phi \rangle \langle \phi |)^{\otimes k}
\end{split}
\label{eqn:core of G(k)}
\end{equation}%
if we take $\mathcal{P}=\{(1,1,0),(2,2,0),(1,0,2),(0,2,2)\}$. 
The set $\mathcal{P}$ is for bookkeeping, incorporating the terms from
\eqref{newlab1} and \eqref{newlab2}, and we remind the readers that
$d \mu_t^{(0)} := d\mu_t$.
We remark
that, to reach \eqref{eqn:core of G(k)}, we used the quantum de Finetti
theorem (i.e., Theorem \ref{Thm:de finetti}) four times: twice for the $%
\gamma ^{(k)}$ term (once for $\gamma_1$ and once for $\gamma_2$) and twice
for the terms in the self-generated potential $\mathcal{A}$ (they are
quadratic in $\rho $).

\subsection{Estimate of $\Tr \left\vert G^{(1)}(t_{1})\right\vert $}

Putting $k=2$ in \eqref{eqn:core of G(k)}, and replacing $\psi, \omega, \phi$
with $\phi_1, \phi_2, \phi_3$, respectively, we have%
\begin{align*}
\Tr \left\vert G^{(1)}(t_{1})\right\vert =& \Tr\left\vert
\int_{0}^{t_{1}}dt_{2}U^{(1)}(t_{1}-t_{2})\left( \mathcal{A}_{1}^{(1)}\gamma
^{(1)}(t_{2})+\mathcal{A}^{(1)}\gamma _{2}^{(1)}(t_{2})\right) \right\vert \\
\leq & \sum_{\left( l,m,n\right) \in \mathcal{P}}\int_{0}^{t_{1}}dt_{2}%
\iiint d\left\vert \mu _{t_{2}}^{(l)}\right\vert (\phi_1 )d\left\vert \mu
_{t_{2}}^{(m)}\right\vert (\phi_2 )d\left\vert \mu _{t_{2}}^{(n)}\right\vert
(\phi_3) \\
&\times \Tr\left\vert U^{(1)}(t_{1}-t_{2})\left[ a_{|\phi_1|^2,
|\phi_2|^2}(r_{1})- a_{|\phi_1|^2, |\phi_2|^2}(r_{1}^{\prime })\right]
\phi_3 \left( r_{1}\right) \bar{\phi}_3\left( r_{1}^{\prime }\right)
\right\vert .
\end{align*}%
Using the fact that 
\begin{align*}
\Tr\left\vert U^{(1)}(t)f\left( r_{1}\right) g\left( r_{1}^{\prime }\right)
\right\vert &=\int \left\vert e^{it\Delta }f(r_{1})e^{-it\Delta
}g(r_{1})\right\vert dx_{1} \\
&\leq \left\Vert e^{it\Delta }f\right\Vert _{L_{x}^{2}}\left\Vert
e^{it\Delta }g\right\Vert _{L_{x}^{2}} \\
&= \left\Vert f\right\Vert _{L_{x}^{2}}\left\Vert g\right\Vert _{L_{x}^{2}},
\end{align*}
we have 
\begin{equation*}
\Tr\left\vert G^{(1)}(t_{1})\right\vert \leq \sum_{\left( l,m,n\right) \in 
\mathcal{P}}\int_{0}^{t_{1}}dt_{2}\iiint d\left\vert \mu
_{t_{2}}^{(l)}\right\vert (\phi_1 )d\left\vert \mu _{t_{2}}^{(m)}\right\vert
(\phi_2)d\left\vert \mu _{t_{2}}^{(n)}\right\vert (\phi_3 )\left\Vert
a_{|\phi_1|^2, |\phi_2|^2}\phi_3 \right\Vert _{L_{x}^{2}}\left\Vert \phi_3
\right\Vert _{L_{x}^{2}}.
\end{equation*}
Corollary \ref{Corollary:Key2}, i.e., the main nonlinear estimate, turns the
above into%
\begin{align*}
\Tr\left\vert G^{(1)}(t_{1})\right\vert \leq &\sum_{\left( l,m,n\right) \in 
\mathcal{P}}\int_{0}^{t_{1}}dt_{2}\iiint d\left\vert \mu
_{t_{2}}^{(l)}\right\vert (\phi_1 )d\left\vert \mu _{t_{2}}^{(m)}\right\vert
(\phi_2 )d\left\vert \mu _{t_{2}}^{(n)}\right\vert (\phi_3 ) \times
\left\Vert \phi_3 \right\Vert _{L_{x}^{2}} \times \\
&\times \| \phi_1 \|_{\dot{H}^{\frac12}_x} \| \phi_2 \|_{\dot{H}%
^{\frac12}_x} \min_{\tau \in S_3} \| \phi_{\tau(1)} \|_{\dot{H}^{\frac12}_x}
\| \phi_{\tau(2)} \|_{\dot{H}^{\frac12}_x} \| \phi_{\tau(3)} \|_{L^2_x}
\end{align*}%
One of $l,m,n$ is zero, and we may put the corresponding term in $L^{2}$,
i.e., 
\begin{align*}
\Tr\left\vert G^{(1)}(t_{1})\right\vert \leq
&\sum_{l=1}^{2}\int_{0}^{t_{1}}dt_{2}\iiint d\mu _{t_{2}}^{(l)}(\phi_1 )d\mu
_{t_{2}}^{(l)}(\phi_2 )d\left\vert \mu _{t_{2}}^{(0)}\right\vert (\phi_3
)\Vert \phi_1 \Vert _{\dot{H}^{\frac12}_x}^{2}\Vert \phi_2 \Vert _{\dot{H}%
^{\frac12}_x}^{2}\left\Vert \phi_3 \right\Vert _{L_{x}^{2}}^{2} \\
& +\int_{0}^{t_{1}}dt_{2}\iiint d\mu _{t_{2}}^{(1)}(\phi_1 )d\left\vert \mu
_{t_{2}}\right\vert (\phi_2 )d\mu _{t_{2}}^{(2)}(\phi_3 )\Vert \phi_1 \Vert
_{\dot{H}^{\frac12}_x}^{2}\Vert \phi_2 \Vert _{L_{x}^{2}} \lVert \phi_2
\rVert_{\dot{H}^{\frac12}_x} \Vert \phi_3 \Vert _{\dot{H}^{\frac12}_x}
\lVert \phi_3 \rVert_{L^2_x} \\
&+\int_{0}^{t_{1}}dt_{2}\iiint d\left\vert \mu _{t_{2}}\right\vert (\phi_1
)d\mu _{t_{2}}^{(2)}(\phi_2 )d\mu _{t_{2}}^{(2)}(\phi_3 ) \Vert \phi_1
\Vert_{L_{x}^{2}} \lVert \phi_1 \rVert_{\dot{H}^{\frac12}_x} \Vert \phi_2
\Vert _{\dot{H}^{\frac12}_x}^{2}\Vert \phi_3 \Vert _{\dot{H}^{\frac12}_x}
\lVert \phi_3 \rVert_{L^2_x}.
\end{align*}
Using the fact that each $\mu_t^{(j)}$ is supported on the unit sphere in $%
L^2$ and thanks to \eqref{measure energy:integral} and 
\eqref{measure
energy:support}, we obtain 
\begin{align*}
\Tr \left\vert G^{(1)}(t_{1})\right\vert &\leq (2 M^2 + 2 M^3) t_1 \sup_{t
\in [0, t_1]} \int d | \mu_{t} |(\phi) \\
&\leq C (M^2 + M^3) t_1\left( \sup_{t\in \left[ 0, t_{1}\right] }\Tr %
\left\vert \gamma^{(1)}(t) \right\vert \right)
\end{align*}
Thus we have proved that 
\begin{equation}
\Tr\left\vert G^{(1)}(t_{1})\right\vert \leq Ct_{1}\left( \sup_{t\in \left[
0,t_{1}\right] }\Tr\left\vert \gamma ^{(1)}(t)\right\vert \right) .
\label{eqn: end of part I}
\end{equation}

\subsection{Estimate for Part $\mathrm{II}$}

Recall that 
\begin{equation*}
\mathrm{II} =\sum_{r=1}^{l_{c}}\int_{0}^{t_{1}}...%
\int_{0}^{t_{r}}dt_{2}...dt_{r+1}U^{(1)}(t_{1}-t_{2})B_{2}...U^{(r)}(t_{r}-t_{r+1})B_{r+1}G^{(r+1)}(t_{r+1}).
\end{equation*}
Because each $B_{j}$ is a sum of $2\left( j-1\right) $ terms (see %
\eqref{Bshort}), integrands of summands of $\mathrm{NP}^{(l_{c})}$ and $%
\mathrm{IP}^{(l_{c})}$ have up to $O(k!)$ summands themselves. We use the
Klainerman-Machedon board game argument to combine them and hence reduce the
number of terms that need to be treated. Define 
\begin{equation*}
J(\underline{t}_{j+1})(f^{(j+1)})=U^{(1)}(t_{1}-t_{2})B_{2}\cdots
U^{(j)}(t_{j}-t_{j+1})B_{j+1}f^{(j+1)},
\end{equation*}
where $\underline{t}_{j+1}$ means $\left( t_{2},\ldots ,t_{j+1}\right)$.
Then the Klainerman-Machedon board game argument implies the lemma.

\begin{lem}[Klainerman-Machedon board game]
\cite{KlMa08}\label{lem:KM}One can express 
\begin{equation*}
\int_{0}^{t_{1}}\cdots \int_{0}^{t_{j}}J(\underline{t}_{j+1})(f^{(j+1)})d%
\underline{t}_{j+1}
\end{equation*}%
as a sum of at most $4^{j}$ terms of the form 
\begin{equation*}
\int_{D}J(\underline{t}_{j+1},\sigma)(f^{(j+1)})d\underline{t}_{j+1},
\end{equation*}%
or in other words, 
\begin{equation*}
\int_{0}^{t_{1}}\cdots \int_{0}^{t_{j}}J(\underline{t}_{j+1})(f^{(j+1)})d%
\underline{t}_{j+1}=\sum_{\sigma}\int_{D}J(\underline{t}_{j+1},%
\sigma)(f^{(j+1)})d\underline{t}_{j+1}.
\end{equation*}%
Here $D\subset \lbrack 0,t_{2}]^{j}$, $\sigma $ belong to the set of maps
from $\{2,\ldots ,j+1\}$ to $\{1,\ldots ,j\}$ satisfying $\sigma (2)=1$ and $%
\sigma(l)<l$ for all $l,$ and 
\begin{align*}
J(\underline{t}_{j+1},\sigma )(f^{(j+1)})=&
\;U^{(1)}(t_{1}-t_{2})B_{1,2}U^{(2)}(t_{2}-t_{3})B_{\sigma (3),3}\cdots \\
& \cdots U^{(j)}(t_{j}-t_{j+1})B_{\sigma (j+1),j+1}(f^{(j+1)}).
\end{align*}
\end{lem}

With Lemma \ref{lem:KM}, we can write a typical summand of part $\mathrm{II}$
as 
\begin{align*}
& \int_{0}^{t_{1}}\cdots \int_{0}^{t_{r}}dt_{2}\cdots
dt_{r+1}U^{(1)}(t_{1}-t_{2})B_{2}\cdots
U^{(r)}(t_{r}-t_{r+1})B_{r+1}G^{(r+1)}(t_{r+1}) \\
& =\sum_{\sigma }\int_{D}d\underline{t}_{r+1}J\left( \underline{t}%
_{r+1},\sigma \right) \left( G^{(r+1)}\right) ,
\end{align*}%
where the sum has at most $4^{r}$ terms inside. Let 
\begin{equation}
\mathrm{II}^{(r,\sigma )}=\int_{D}d\underline{t}_{r+1}J\left( \underline{t}%
_{r+1},\sigma \right) \left( G^{(r+1)}\right)  \label{partIIexp}
\end{equation}%
To estimate part $\mathrm{II}$, it suffices to prove the following lemma.

\begin{lem}
\label{lem:individual summand}There is a $C$ depending on $M$ in %
\eqref{energy condition} such that for all $r$, we have 
\begin{equation*}
\Tr\left\vert \mathrm{II}^{(r,\sigma )}\right\vert (t_{1})\leq \left[ \left(
r+1\right) \left( C_{0}t_{1}^{\frac{1}{3}}\right) ^{r}\right] t_{1}\left(
\sup_{t\in \left[ 0,t_{1}\right] }\Tr\left\vert \gamma ^{(1)}(t)\right\vert
\right)
\end{equation*}
\end{lem}

With the above lemma, we have%
\begin{equation}  \label{eqn: end of part II}
\begin{split}
\Tr\left\vert \mathrm{II} \right\vert (t_{1}) &\leq
\sum_{r=1}^{l_{c}}\sum_{\sigma}\left[ \left( r+1\right) \left( C_{0}t_{1}^{%
\frac{1}{3}}\right) ^{r}\right] t_{1}\left( \sup_{t\in \left[ 0,t_{1}\right]
}\Tr\left\vert \gamma ^{(1)}(t)\right\vert \right) \\
&\leq t_{1}\left( \sup_{t\in \left[ 0,t_{1}\right] }\Tr\left\vert \gamma
^{(1)}(t)\right\vert \right) \sum_{r=1}^{\infty }4^{r}\left[ \left(
r+1\right) \left( C_{0}t_{1}^{\frac{1}{3}}\right) ^{r}\right] \\
&\leq Ct_{1}\left( \sup_{t\in \left[ 0,t_{1}\right] }\Tr\left\vert \gamma
^{(1)}(t)\right\vert \right) ,
\end{split}%
\end{equation}
for $t_{1}$ small enough so that the series converges.

Together the estimates \eqref{eqn: end of part I} and 
\eqref{eqn: end of
part II} establish Theorem \ref{Thm:estimate for NP}.

Before proving Lemma \ref{lem:individual summand}, we illustrate how to
obtain the estimate for a specific example.

\begin{examp}
\label{example:3B}To avoid heavy notation and demonstrate the main idea of
the proof of Lemma \ref{lem:individual summand}, we first prove it for a
concrete example. The general case uses the same underlying idea, which
turns out to be quite simple as compared to what must be done for the
interaction part $\mathrm{IP}$. We adapt the example and use the notation in 
\cite[\S 6.1]{ChHaPaSe13} for our $\mathrm{II}^{(r,\sigma )}$. Denoting
$U^{(j)}(t_{k}-t_{l})$ by $U_{k,l}^{(j)}$, we consider 
\begin{equation}
\begin{split}
\Tr\left\vert \mathrm{II}^{(3,\sigma )}\right\vert (t_{1})=& \int_{D}d%
\underline{t}%
_{4}U_{1,2}^{(1)}B_{1,2}U_{2,3}^{(2)}B_{2,3}U_{3,4}^{(3)}B_{3,4}[G^{(4)}(t_{4})]
\\
\leq & \sum_{j=1}^{4}\sum_{\left( l,m,n\right) \in \mathcal{P}%
}\int_{[0,t_{1}]^{3}}d\underline{t}_{4}\int_{0}^{t_{1}}dt_{5}\iiint
d\left\vert \mu _{t_{5}}^{(l)}\right\vert (\psi )d\left\vert \mu
_{t_{5}}^{(m)}\right\vert (\omega )d\left\vert \mu _{t_{5}}^{(n)}\right\vert
(\phi ) \\
& \Tr%
|U_{1,2}^{(1)}B_{1,2}U_{2,3}^{(2)}B_{2,3}U_{3,4}^{(3)}B_{3,4}U_{4,5}^{(4)}%
\left( \left[ a_{|\psi |^{2},|\omega |^{2}}(r_{j})-a_{|\psi |^{2},|\omega
|^{2}}(r_{j}^{\prime })\right] (|\phi \rangle \langle \phi |)^{\otimes
4}\right) |
\end{split}
\label{e:example e1}
\end{equation}

\begin{rem}
In the above, there is a $U_{4,5}^{(4)}$ after $B_{3,4}$. This is the main
difference between the nonlinear part $\mathrm{NP}$ and the interaction part 
$\mathrm{IP}$. As noted in \cite{ChHaPaSe13}, since the last $B$ in $\mathrm{%
IP}$ is not followed by a Schr\"{o}dinger propagator, it creates a factor of 
$\left\vert \phi \right\vert ^{2}\phi$, which has to be handled by Sobolev
embedding rather than Strichartz estimates.
\end{rem}

It suffices to treat%
\begin{equation}  \label{e:example e2}
\begin{split}
&\sum_{\left( l,m,n\right) \in \mathcal{P}}\int_{[0,t_{1}]^{3}}d\underline{t}%
_{4}\int_{0}^{t_{1}}dt_{5}\iiint d\left\vert \mu _{t_{5}}^{(l)}\right\vert
(\psi )d\left\vert \mu _{t_{5}}^{(m)}\right\vert (\omega )d\left\vert \mu
_{t_{5}}^{(n)}\right\vert (\phi ) \\
&\Tr%
|U_{1,2}^{(1)}B_{1,2}^{+}U_{2,3}^{(2)}B_{2,3}^{+}U_{3,4}^{(3)}B_{3,4}^{+}U_{4,5}^{(4)}\left( %
\left[ a_{|\psi |^{2},|\omega |^{2}}(r_{4})\right] (|\phi \rangle \langle
\phi |)^{\otimes 4}\right) |
\end{split}%
\end{equation}
where $B_{1,2}^{+}$ is half of $B_{1,2}$, namely 
\begin{equation*}
B_{1,2}^{+}(\gamma ^{(2)})=\gamma ^{(2)}\left( x_{1},x_{1},x_{1}^{\prime
},x_{1}\right) .
\end{equation*}%
When we plug the estimate of \eqref{e:example e2} into \eqref{e:example e1},
we will pick up a $2^{3}$ since there are three $B$'s in \eqref{e:example e1}%
. However, compensating for this is the factor $\left( t_{1}^{\frac{2}{3}%
}\right) ^{3}$ that emerges by the end. Hence our simplification is a valid
one.

Step I. (Structure) We enumerate the four factors of $(|\phi \rangle \langle
\phi |)^{\otimes 4}$ for the purpose of bookkeeping, even though these
factors are physically indistinguishable. So we write $\otimes
_{i=1}^{4}u_{i}$, ordered with increasing index $i$. We first have%
\begin{equation*}
B_{3,4}^{+}U_{4,5}^{(4)}a_{|\psi |^{2},|\omega |^{2}}(r_{4})(|\phi \rangle
\langle \phi |)^{\otimes 4}=\left( U_{4,5}^{(2)}\left( \otimes
_{i=1}^{2}u_{i}\right) \right) \otimes \Theta _{3},
\end{equation*}%
where%
\begin{align}
\Theta _{3}& =B_{1,2}^{+}(U_{4,5}^{(2)}(u_{3}\otimes a_{|\psi |^{2},|\omega
|^{2}}(r_{4})u_{4}))  \notag \\
& =B_{1,2}^{+}\left( U_{4,5}\phi (x_{3})\right) \left( U_{5,4}\bar{\phi}%
(x_{3}^{\prime })\right) \left( U_{4,5}\left[ a_{|\psi |^{2},|\omega
|^{2}}(r_{4})\phi (x_{4})\right] \right) \left( U_{5,4}\bar{\phi}%
(x_{4}^{\prime })\right)  \notag \\
& =\left( U_{4,5}\phi (x_{3})\right) \left( U_{4,5}\left[ a_{|\psi
|^{2},|\omega |^{2}}(r_{3})\phi (x_{3})\right] \right) \left( U_{5,4}\bar{%
\phi}(x_{3})\right) \left( U_{5,4}\bar{\phi}(x_{3}^{\prime })\right)  \notag
\\
& \equiv T_{3}(x_{3})\left( U_{5,4}\bar{\phi}(x_{3}^{\prime })\right)
\label{eqn:T3}
\end{align}%
with $U_{4,5}=e^{i\left( t_{4}-t_{5}\right) \Delta }$. 
Here $T_3$ stands for the trilinear form
\[
\left( U_{4,5}\left[ a_{|\psi
|^{2},|\omega |^{2}}(r_{3})\phi (x_{3})\right] \right) \left( U_{5,4}\bar{%
\phi}(x_{3})\right) \left( U_{5,4}\bar{\phi}(x_{3}^{\prime })\right)
\]
We make similar substitutions below and, to bound these terms, 
shall invoke the trilinear estimate \eqref{TL2}, which states that
\[
\| T(f_1, f_2, f_3) \|_{L^1_{t \in [0, t_0)} L^2_x} 
\lesssim_s t_0^{\frac{s}{2}} \| f_1 \|_{L^2} \| f_2 \|_{L^2} \| f_3 \|_{\dot{H}^s}
\]
for $0 < s \leq 2$.

Applying $%
B_{2,3}^{+}U_{3,4}^{(3)}$, we reach 
\begin{align*}
& B_{2,3}^{+}U_{3,4}^{(3)}B_{3,4}^{+}U_{4,5}^{(4)}\left( a_{\rho _{l},\rho
_{m}}(r_{j})(|\phi \rangle \langle \phi |)^{\otimes 4}\right) \\
& =B_{2,3}^{+}U_{3,4}^{(3)}\left( U_{4,5}^{(1)}u_{1}\otimes
U_{4,5}^{(1)}u_{2}\otimes \Theta _{3}\right) \\
& =U_{3,4}^{(1)}U_{4,5}^{(1)}u_{1}\otimes \Theta _{2} \\
& =U_{3,5}^{(1)}u_{1}\otimes \Theta _{2}
\end{align*}%
where 
\begin{align}
\Theta _{2}& =B_{1,2}^{+}U_{3,4}^{(2)}\left( U_{4,5}^{(1)}u_{2}\otimes
\Theta _{3}\right)  \notag \\
& =B_{1,2}^{+}\left( U_{3,5}^{(1)}u_{2}\otimes U_{3,4}^{(1)}\Theta
_{3}\right)  \notag \\
& =B_{1,2}^{+}\left( \left( U_{3,5}\phi (x_{2})\right) \left( U_{5,3}\bar{%
\phi}(x_{2}^{\prime })\right) \left( U_{3,4}T_{3}(x_{3})\right) \left(
U_{4,3}U_{5,4}\bar{\phi}(x_{3}^{\prime })\right) \right)  \notag \\
& =\left( U_{3,5}\phi (x_{2})\right) (U_{3,4}T_{3}(x_{2}))\left( U_{5,3}\bar{%
\phi}(x_{2})\right) \left( U_{5,3}\bar{\phi}(x_{2}^{\prime })\right)  \notag
\\
& \equiv T_{2}(x_{2})\left( U_{5,3}\bar{\phi}(x_{2}^{\prime })\right) .
\label{eqn:T2}
\end{align}%
Finally, with $U_{1,2}^{(1)}B_{1,2}^{+}U_{2,3}^{(2)}$, we get%
\begin{align}
&
U_{1,2}^{(1)}B_{1,2}^{+}U_{2,3}^{(2)}B_{2,3}^{+}U_{3,4}^{(3)}B_{3,4}^{+}U_{4,5}^{(4)}\left( a_{|\psi |^{2},|\omega |^{2}}(r_{j})(|\phi \rangle \langle \phi |)^{\otimes 4}\right)
\notag \\
& =U_{1,2}^{(1)}B_{1,2}^{+}U_{2,3}^{(2)}\left( U_{3,5}^{(1)}u_{1}\otimes
\Theta _{2}\right)  \notag \\
& =U_{1,2}^{(1)}B_{1,2}^{+}\left( U_{2,5}^{(1)}u_{1}\otimes
U_{2,3}^{(1)}\Theta _{2}\right)  \notag \\
& =U_{1,2}^{(1)}B_{1,2}^{+}\left[ \left( U_{2,5}\phi (x_{1})\right) \left(
U_{5,2}\bar{\phi}(x_{1}^{\prime })\right) \left( U_{2,3}T_{2}(x_{2})\right)
\left( U_{3,2}U_{5,3}\bar{\phi}(x_{2}^{\prime })\right) \right]  \notag \\
& =U_{1,2}^{(1)}\left[ \left( U_{2,5}\phi (x_{1})\right) \left(
U_{2,3}T_{2}(x_{1})\right) \left( U_{5,2}\bar{\phi}(x_{1})\right) \left(
U_{5,2}\bar{\phi}(x_{1}^{\prime })\right) \right]  \notag \\
& =U_{1,2}^{(1)}\left[ T_{1}(x_{1})U_{5,2}\bar{\phi}(x_{1}^{\prime })\right]
.  \label{eqn:T1}
\end{align}

Step II. (Iterative Estimate) Plugging the calculation in Step I into %
\eqref{e:example e2}, we have%
\begin{align*}
\eqref{e:example e2}\leq & \sum_{\left( l,m,n\right) \in \mathcal{P}%
}\int_{[0,t_{1}]^{3}}d\underline{t}_{4}\int_{0}^{t_{1}}dt_{5}\iiint
d\left\vert \mu _{t_{5}}^{(l)}\right\vert (\psi )d\left\vert \mu
_{t_{5}}^{(m)}\right\vert (\omega )d\left\vert \mu _{t_{5}}^{(n)}\right\vert
(\phi ) \\
& \times \left\Vert T_{1}(x_{1})\right\Vert _{L^{2}}\left\Vert \bar{\phi}%
\right\Vert _{L^{2}} \\
\leq & \sum_{\left( l,m,n\right) \in \mathcal{P}%
}\int_{[0,t_{1}]^{2}}dt_{3}t_{4}\int_{0}^{t_{1}}dt_{5}\iiint d\left\vert \mu
_{t_{5}}^{(l)}\right\vert (\psi )d\left\vert \mu _{t_{5}}^{(m)}\right\vert
(\omega )d\left\vert \mu _{t_{5}}^{(n)}\right\vert (\phi ) \\
& \times \left\Vert T_{1}\right\Vert _{L_{t_{2}}^{1}L^{2}}
\end{align*}%
where 
\begin{equation*}
\left\Vert T_{1}\right\Vert _{L_{t_{2}}^{1}L^{2}}\leq Ct_{1}^{\frac{1}{3}%
}\left\Vert \phi \right\Vert _{H^{\frac{2}{3}}}\left\Vert T_{2}\right\Vert
_{L^{2}}\left\Vert \phi \right\Vert _{L^{2}}
\end{equation*}%
by \eqref{TL2}. 
Thus%
\begin{align*}
\eqref{e:example e2}\leq & \;Ct_{1}^{\frac{1}{3}}\sum_{\left( l,m,n\right)
\in \mathcal{P}}\int_{[0,t_{1}]}t_{4}\int_{0}^{t_{1}}dt_{5}\iiint
d\left\vert \mu _{t_{5}}^{(l)}\right\vert (\psi )d\left\vert \mu
_{t_{5}}^{(m)}\right\vert (\omega )d\left\vert \mu _{t_{5}}^{(n)}\right\vert
(\phi ) \\
& \left\Vert \phi \right\Vert _{H^{\frac{2}{3}}}\left\Vert T_{2}\right\Vert
_{L_{t_{3}}^{1}L^{2}}.
\end{align*}%
By \eqref{TL2} again, 
\begin{equation*}
\left\Vert T_{2}(x_{2})\right\Vert _{L_{t_{3}}^{1}L_{x_{2}}^{2}}\leq Ct_{1}^{%
\frac{1}{3}}\left\Vert \phi \right\Vert _{H^{\frac{2}{3}}}\left\Vert
T_{3}\right\Vert _{L^{2}}\left\Vert \phi \right\Vert _{L^{2}},
\end{equation*}%
and hence 
\begin{align*}
\eqref{e:example e2}\leq & \left( Ct_{1}^{\frac{1}{3}}\right)
^{2}\sum_{\left( l,m,n\right) \in \mathcal{P}}\int_{0}^{t_{1}}dt_{5}\iiint
d\left\vert \mu _{t_{5}}^{(l)}\right\vert (\psi )d\left\vert \mu
_{t_{5}}^{(m)}\right\vert (\omega )d\left\vert \mu _{t_{5}}^{(n)}\right\vert
(\phi ) \\
& \left\Vert \phi \right\Vert _{H^{\frac{2}{3}}}^{2}\left\Vert
T_{3}\right\Vert _{L_{t_{4}}^{1}L_{x_{2}}^{2}} \\
\leq & \left( Ct_{1}^{\frac{1}{3}}\right) ^{3}\sum_{\left( l,m,n\right) \in 
\mathcal{P}}\int_{0}^{t_{1}}dt_{5}\iiint d\left\vert \mu
_{t_{5}}^{(l)}\right\vert (\psi )d\left\vert \mu _{t_{5}}^{(m)}\right\vert
(\omega )d\left\vert \mu _{t_{5}}^{(n)}\right\vert (\phi ) \\
& \left\Vert \phi \right\Vert _{H^{\frac{2}{3}}}^{3}\left\Vert a_{|\psi
|^{2},|\omega |^{2}}(r_{3})\phi (x_{3})\right\Vert _{L^{2}}.
\end{align*}%
By the fact that $\left\vert \mu _{t}^{(i)}\right\vert $ is supported in the
set 
\begin{equation*}
\left\{ \phi \in L^{2}(\mathbb{R}^{2})|\left\Vert \phi \right\Vert _{H^{%
\frac{2}{3}}}\leq M\right\} ,
\end{equation*}%
we have%
\begin{align*}
\eqref{e:example e2}\leq & \left( CMt_{1}^{\frac{1}{3}}\right)
^{3}\sum_{\left( l,m,n\right) \in \mathcal{P}}\int_{0}^{t_{1}}dt_{5}\iiint
d\left\vert \mu _{t_{5}}^{(l)}\right\vert (\psi )d\left\vert \mu
_{t_{5}}^{(m)}\right\vert (\omega )d\left\vert \mu _{t_{5}}^{(n)}\right\vert
(\phi ) \\
& \left\Vert a_{|\psi |^{2},|\omega |^{2}}(r_{3})\phi (x_{3})\right\Vert
_{L^{2}}.
\end{align*}%
One then proceeds as in the estimate of $\Tr\left\vert
G^{(1)}(t_{1})\right\vert $ to reach 
\begin{equation*}
\eqref{e:example e2}\leq \left( CMt_{1}^{\frac{1}{3}}\right)
^{3}M^{4}t_{1}\left( \sup_{t\in \left[ 0,t_{1}\right] }\Tr\left\vert \gamma
^{(1)}(t)\right\vert \right)
\end{equation*}%
Selecting a $C_{0}$ bigger than $M$ and $1$, we obtain%
\begin{equation*}
\eqref{e:example e2}\leq \left( C_{0}t_{1}^{\frac{1}{3}}\right)
^{3}t_{1}\left( \sup_{t\in \left[ 0,t_{1}\right] }\Tr\left\vert \gamma
^{(1)}(t)\right\vert \right) .
\end{equation*}%
Plugging the above estimate back into \eqref{e:example e1}, we get%
\begin{equation*}
\Tr\left\vert \mathrm{II}^{(3,\sigma )}\right\vert (t_{1})\leq \left[ 4\cdot
2^{3}\cdot \left( C_{0}t_{1}^{\frac{1}{3}}\right) ^{3}\right] t_{1}\left(
\sup_{t\in \left[ 0,t_{1}\right] }\Tr\left\vert \gamma ^{(1)}(t)\right\vert
\right)
\end{equation*}%
as desired. This finishes the proof of the example.
\end{examp}

One observation to make concerning our approach in Example \ref{example:3B}
is that the structure found in Step I is crucial. Such a structure generated
by the collision operator $B$ and propagator $U$ is found in general, and we
state its relevant properties in the following lemma.

\begin{lem}
\label{lem:tricol} Let $M\in \mathbb{N}$, $M>1$, and for each $j$, $1\leq
j\leq M$, suppose that the two functions $f_{j}(x_{j})$, $f_j^{\prime
}(x_{j}^{\prime })$ belong to\footnote{We suppress the time dependence in the notation
and allow restriction to time intervals, which may be achieved, for instance, by
introducing sharp time cutoffs.} $L^1_t H^{s}_x(\mathbb{R}^{2})$, $\frac{1}{2}\leq
s\leq \frac{2}{3}$.
Then there exist $L^1_t H^{s}_x(\mathbb{R}^{2})$ functions $%
h,h^{\prime }$ such that 
\begin{equation*}
B_{\sigma (M),M}^{\pm }U_{M,M+1}^{(M)}\left[ \prod_{j=1}^{M}f_{j}(x_{j})%
f_{j}^{\prime }(x_{j}^{\prime })\right] =h_{\sigma (M)}(x_{\sigma
(M)})h^{\prime }_{\sigma (M)}(x_{\sigma (M)}^{\prime
})U_{M,M+1}^{(M-2)}\left[ \prod_{\substack{j=1 \\ j \neq \sigma(M)}}^{M-1}f_{j}(x_{j})f_{j}^{\prime
}(x_{j}^{\prime })\right]
\end{equation*}%
In the case where $B$ is $B_{\sigma (M),M}^{+}$, $h$ is a trilinear form of
the type \eqref{T} and $h^{\prime }$ is a linear evolution. In the case
where $B$ is $B_{\sigma (M),M}^{-}$, the roles of $h$ and $h^{\prime }$ are
reversed.
\end{lem}
\begin{proof}
The collision operator leaves untouched each term for which $j\notin
\{M,\sigma (M)\}$. Only the propagator affects these terms. So we have 
\begin{equation*}
\begin{split}
& B_{\sigma (M),M}^{+}U_{M,M+1}^{(M)}\left[ \prod_{j=1}^{M}f_{j}(x_{j})%
f_{j}^{\prime }(x_{j}^{\prime })\right] \\
& =U_{M,M+1}^{(M-2)}\left[ \prod_{j\in \{1,\ldots ,M\}\setminus \{M,\sigma
(M)\}}f_{j}(x_{j})f_{j}^{\prime }(x_{j}^{\prime })\right] \cdot
T_{\sigma (M),M}(x_{\sigma (M)})e^{-i(t_{M}-t_{M+1})\Delta _{x_{\sigma
(M)}^{\prime }}}f^{\prime }_{\sigma (M)}(x_{\sigma (M)}^{\prime })
\end{split}%
\end{equation*}%
where 
\begin{equation*}
\begin{split}
T_{\sigma (M),M}(x_{\sigma (M)}):=& \;e^{i(t_{M}-t_{M+1})\Delta _{x_{\sigma
(M)}}}f_{\sigma (M)}(x_{\sigma (M)})\cdot e^{i(t_{M}-t_{M+1})\Delta
_{x_{\sigma (M)}}}f_{M}(x_{\sigma (M)})\cdot \\
& \quad \cdot e^{-i(t_{M}-t_{M+1})\Delta _{x_{\sigma (M)}}}
f^{\prime }_{M}(x_{\sigma (M)})
\end{split}%
\end{equation*}

Similarly, 
\begin{equation*}
\begin{split}
& B_{\sigma (M),M}^{-}U_{M,M+1}^{(M)}\left[ \prod_{j=1}^{M}f_{j}(x_{j})%
f_{j}^{\prime }(x_{j}^{\prime })\right] \\
& =U_{M,M+1}^{(M-2)}\left[ \prod_{j\in \{1,\ldots ,M\}\setminus \{M,\sigma
(M)\}}f_{j}(x_{j})f_{j}^{\prime }(x_{j}^{\prime })\right] \cdot
T_{\sigma (M),M}^{\prime }(x_{\sigma (M)}^{\prime
})e^{i(t_{M}-t_{M+1})\Delta _{x_{\sigma (M)}}}f_{\sigma (M)}(x_{\sigma (M)})
\end{split}%
\end{equation*}%
where 
\begin{equation*}
\begin{split}
T_{\sigma (M),M}^{\prime }(x_{\sigma (M)}^{\prime }):=&
\;e^{i(t_{M}-t_{M+1})\Delta _{x_{\sigma (M)}^{\prime }}}f_{M}(x_{\sigma
(M)}^{\prime })\cdot e^{-i(t_{M}-t_{M+1})\Delta _{x_{\sigma (M)}^{\prime }}}%
f^{\prime }_{\sigma (M)}(x_{\sigma (M)}^{\prime })\cdot \\
& \quad \cdot e^{-i(t_{M}-t_{M+1})\Delta _{x_{\sigma (M)}^{\prime }}}%
f^{\prime }_{M}(x_{\sigma (M)}^{\prime })
\end{split}%
\end{equation*}

The $L^1_t H^{s}_x$ bounds follow from \eqref{TH} and Strichartz.
\end{proof}

\begin{proof}[Proof of Lemma \protect\ref{lem:individual summand}]
Using \eqref{partIIexp}, \eqref{Gk}, and \eqref{eqn:core of G(k)}, we write 
\begin{equation*}
\begin{split}
\mathrm{II}^{(r,\sigma )}=& \sum_{j=1}^{r+1}\sum_{(l,m,n)\in \mathcal{P}%
}\int_{D}d\underline{t}_{r+1}J(\underline{t}_{r+1},\sigma )\left\{
\int_{0}^{t_{r+1}}dt_{r+2}U^{(r+1)}(t_{r+1}-t_{r+2})\right. \\
& \left. \iiint d\mu _{t_{r+2}}^{(l)}(\psi )d\mu _{t_{r+2}}^{(m)}(\omega
)d\mu _{t_{r+2}}^{(n)}(\phi )\left[ a_{|\psi |^{2},|\omega
|^{2}}(|x_{j}|)-a_{|\psi |^{2},|\omega |^{2}}(|x_{j}^{\prime }|)\right]
(|\phi \rangle \langle \phi |)^{\otimes (r+1)}\right\}
\end{split}%
\end{equation*}%
We abbreviate 
\begin{equation*}
J(\underline{t}_{r+1},\sigma )=U_{1,2}^{(1)}B_{1,2}U_{2,3}^{(2)}B_{\sigma
(3),3}\cdots U_{r,r+1}^{(r)}B_{\sigma (r+1),r+1}
\end{equation*}%
and write 
\begin{equation*}
\begin{split}
\Tr &\left\vert \mathrm{II}^{(r,\sigma )}\right\vert (t_{1})
\\
\leq
&\sum_{j=1}^{r+1}\sum_{(l,m,n)\in \mathcal{P}}\int_{[0,t_{1}]^{r}}d\underline{%
t}_{r+1}\int_{0}^{t_{1}}dt_{r+2}\iiint d\left\vert \mu
_{t_{r+2}}^{(l)}\right\vert (\psi )d\left\vert \mu
_{t_{r+2}}^{(m)}\right\vert (\omega )d\left\vert \mu
_{t_{r+2}}^{(n)}\right\vert (\phi ) \\
& \Tr\left\vert U_{1,2}^{(1)}B_{1,2}\cdots U_{r,r+1}^{(r)}B_{\sigma
(r+1),r+1}U_{r+1,r+2}^{(r+1)}\left[ a_{|\psi |^{2},|\omega
|^{2}}(|x_{j}|)-a_{|\psi |^{2},|\omega |^{2}}(|x_{j}^{\prime }|)\right]
(|\phi \rangle \langle \phi |)^{\otimes (r+1)}\right\vert
\end{split}%
\end{equation*}%
To simplify calculations, we drop, without loss of generality, the $%
-a_{|\psi |^{2},|\omega |^{2}}(|x_{j}^{\prime }|)$ term. Also, we split each 
$B_{j,k}$ into two pieces $B_{j,k}^{\pm }$ so that $%
B_{j,k}=B_{j,k}^{+}-B_{j,k}^{-}$. 

Consider first the innermost terms 
\begin{equation*}
B_{\sigma (r+1),r+1}^{\pm }U_{r+1,r+2}^{(r+1)}a_{|\psi |^{2},|\omega
|^{2}}(|x_{j}|)(|\phi \rangle \langle \phi |)^{\otimes (r+1)}
\end{equation*}%
The index $j\in \{1,\ldots ,r+1\}$ and the permutation $\sigma $ together
determine at what point $a_{|\psi |^{2},|\omega |^{2}}(|x_{j}|)$ is directly
affected by a collision operator. In any case, we claim that, with respect
to the variables $x_{\sigma (r+1)},x_{\sigma (r+1)}^{\prime }$, the term 
\begin{equation*}
B_{\sigma (r+1),r+1}^{+}U_{r+1,r+2}^{(r+1)}a_{|\psi |^{2},|\omega
|^{2}}(|x_{j}|)(|\phi \rangle \langle \phi |)^{\otimes (r+1)}
\end{equation*}%
is a trilinear form of the form $T$ in \eqref{T} (see \eqref{eqn:T3}, %
\eqref{eqn:T2}, \eqref{eqn:T1} for examples of these trilinear forms) in the 
$x_{\sigma (r+1)}$ variable and a linear flow in the $x_{\sigma
(r+1)}^{\prime }$ variable (the term with $B^-$ instead of $B^+$ is similar
but with the roles of the primed and unprimed variables reversed). Note that
precisely one of the terms in the trilinear form $T$ involves $a_{|\psi
|^{2},|\omega |^{2}}(|x_{j}|)$. This follows from Lemma \ref{lem:tricol}.
Additionally, Lemma \ref{lem:tricol} is formulated so that we can apply it
iteratively until termination, at which point we have one term that is
trilinear of the form \eqref{T} in precisely one of $x_{1}$, $x_{1}^{\prime }
$, and another term that is a linear evolution of a function of the
remaining spatial variable. Step I of Example \ref{example:3B} illustrates
such a process.

The final step is to iteratively bound the terms. We follow Step II of
Example \ref{example:3B}. The underlying idea behind the iterative bounds is
relatively straightforward. We start by controlling the trace norm using
Cauchy-Schwarz in space. One factor is simply a $\phi $ term associated to
the measure, and so will have $L^{2}$ norm equal to one. This leaves us with
the other term in $L_{t}^{1}L^{2}$. The next step is to apply \eqref{TL2}.
This places one factor in $\dot{H}^{s}$ and the remaining ones in $L^{2}$.
So that we can eventually apply \eqref{mainnest}, it is important to always
place in $L^{2}$ the term appearing in the right-hand side that involves $%
a_{|\psi |^{2},|\omega |^{2}}(|x_{j}|)$. To control the term placed in $\dot{%
H}^{s}$, we apply \eqref{TH}. For the terms in $L^{2}$, we use \eqref{TL2}
or \eqref{mainnest} as appropriate.
\end{proof}

\begin{rem}
We first remind the reader that, because at each step we are estimating a
linear term of the type $e^{it\Delta }f$ or a trilinear term of the form %
\eqref{T}, we do not need to apply Sobolev embedding as is necessary for
estimating the interaction part. Secondly, the ``$a$" term cannot be
generated by $B$, and thus we do not need to keep track of multiple
``copies" of $\left\vert \phi \right\vert ^{2}\phi $ generated by $B$, in
contrast to what must be done in controlling the interaction part. In
particular, there is no need to introduce binary tree graphs or keep track
of complicated factorization structures of kernels in controlling the
nonlinear part.
\end{rem}

\section{Multilinear estimates}

\label{sec:multi}

In this section we will need the following fractional Leibniz rule
from \cite[Prop.~3.3]{CrWe91}:

\begin{lem}
\label{lem:ProductRule} Let $0 < s \leq 1$ and $1 < r, p_1, p_2, q_1, q_2 <
\infty$ such that $\frac1r = \frac{1}{p_i} + \frac{1}{q_1}$ for $i = 1, 2$.
Then 
\begin{equation*}
\| |\nabla|^s (fg) \|_{L^r} \lesssim \| f \|_{L^{p_1}} \| |\nabla|^s g
\|_{L^{q_1}} + \| |\nabla|^s f \|_{L^{p_2}} \| g \|_{L^{q_2}}
\end{equation*}
\end{lem}

Define the trilinear form $\cT$ by 
\begin{equation}  \label{T}
\cT(f, g, h) = e^{i (t - t_1) \Delta} f \cdot e^{i (t - t_2) \Delta} g \cdot
e^{i (t - t_3) \Delta} h
\end{equation}

\begin{lem}
\footnote{%
Such trilinear estimates are the precursors to the Klainerman-Machedon
collapsing estimates widely used in the literature. For those estimates, see 
\cite{KlMa08, KiScSt11,
GM,TChenAndNP,ChenDie,ChenAnisotropic,Beckner,GrSoSt14}.}Let $0<s\leq \frac{2%
}{3}$. The trilinear form $\cT$ given by \eqref{T} satisfies 
\begin{equation}
\Vert \cT(f,g,h)\Vert _{L_{t\in \lbrack 0, t_0)}^{1}\dot{H}_{x}^{s}}\lesssim
t_0^{s}\Vert f\Vert _{\dot{H}^{s}}\Vert g\Vert _{\dot{H}^{s}}\Vert h\Vert _{%
\dot{H}^{s}}  \label{TH}
\end{equation}
\end{lem}

\begin{proof}
By the fractional Leibniz rule, we have 
\begin{equation*}
\begin{split}
\| \cT(f, g, h) \|_{L^1_t \dot{H}^{s}_x} \lesssim& \; \| e^{i(t - t_1) \Delta}
f \|_{L^3_t \dot{W}^{s, 6}_x} \| e^{i(t - t_2) \Delta} g \|_{L^3_t L^6_x} \|
e^{i (t -t_3) \Delta} h \|_{L^3_t L^6_x} \\
&+ \| e^{i(t - t_1) \Delta} f \|_{L^3_t L^6_x} \| e^{i(t - t_2) \Delta} g
\|_{L^3_t \dot{W}^{s, 6}_x} \|e^{i(t-t_3)\Delta}h\|_{L^3_t L^6_x} \\
&+ \| e^{i(t - t_1) \Delta} f \|_{L^3_t L^6_x} \| e^{i(t - t_2) \Delta} g
\|_{L^3_t L^6_x} \|e^{i(t-t_3)\Delta}h\|_{L^3_t \dot{W}^{s,6}_x}
\end{split}%
\end{equation*}
By Sobolev embedding, we bound the first term by 
\begin{equation*}
\| e^{i(t - t_1) \Delta} f \|_{L^3_t \dot{W}^{s, 6}_x} \| e^{i(t - t_2)
\Delta} g \|_{L^3_t \dot{W}^{s, p}_x} \| e^{i (t -t_3) \Delta} h \|_{L^3_t 
\dot{W}^{s, p}_x}
\end{equation*}
where $\frac1p = \frac16 + \frac{s}{2}$. Note that $2 \leq p < 6$. Let $q$
be given by $\frac1q + \frac1p = \frac12$ so that $(q, p)$ form a
Schr\"odinger-admissible Strichartz pair (see, for instance, \cite[\S 2]%
{Tao06}). So we use H\"older in time to bound the expression by 
\begin{equation*}
\| e^{i(t - t_1) \Delta} f \|_{L^3_t \dot{W}^{s, 6}_x} t_0^{\frac13 - \frac1q}
\| e^{i(t - t_2) \Delta} g \|_{L^{q}_t \dot{W}^{s, p}_x} t_0^{\frac13 -
\frac1q} \| e^{i (t -t_3) \Delta} h \|_{L^{q}_t \dot{W}^{s, p}_x}
\end{equation*}
Finally, we conclude by applying Strichartz estimates and noting that $%
\frac13 - \frac1q = \frac{s}{2}$. The second and third terms are similar.
\end{proof}

\begin{lem}
Let $0 < s \leq 2$. The trilinear form $\cT$ given by \eqref{T} satisfies 
\begin{equation}  \label{TL2}
\| \cT(f, g, h) \|_{L^1_{t \in [0, t_0)} L^2_x} \lesssim t_0^{\frac{s}{2}} \| f
\|_{L^2} \| g \|_{L^2} \| h \|_{\dot{H}^{s}}
\end{equation}
\end{lem}

\begin{proof}
By H\"older's inequality, 
\begin{equation*}
\| T(f, g, h) \|_{L^1_{t \in [0, t_0)} L^2_x} \leq t_0^{\frac{s}{2}} \| e^{i (t
- t_1) \Delta} f \|_{L^{q}_t L^{r}_x} \| e^{i (t - t_2) \Delta} g
\|_{L^{q}_t L^{r}_x} \| e^{i (t - t_3) \Delta} h \|_{L^\infty_t L^{p}_x}
\end{equation*}
where $\frac1q = \frac12 - \frac{s}{4}$, $r = 4/s$, and $p = 2/(1 - s)$.
Using Strichartz estimates and Sobolev embedding, we control the right hand
side by 
\begin{equation*}
t_0^{\frac{s}{2}} \| f \|_{L^2_x} \| g \|_{L^2_x} \| e^{i (t - t_3) \Delta} h
\|_{L^\infty_t \dot{H}^{s}_x}
\end{equation*}
Finally, we conclude the bound stated in the lemma by noting that the
Schr\"odinger propagator is an isometry on $L^2$-based spaces.
\end{proof}

\begin{rem}
From the proofs of both \eqref{TH} and \eqref{TL2} it is evident that any of 
$e^{i (t - t_1) \Delta}f$, $e^{i (t - t_2) \Delta} g$, and $e^{i (t - t_3)
\Delta} h$ can be replaced by its complex conjugate in the trilinear form %
\eqref{T}.
\end{rem}

For the next set of estimates, recall 
\begin{equation*}
\partial_r A_0 = \frac1r A_\theta \rho, \quad \partial_r A_\theta = -
\frac12 r \rho
\end{equation*}
and 
\begin{equation}  \label{Adef}
A_0(t, r) := -
\int_r^\infty \frac{A_\theta(s)}{s} \rho(s) ds, \quad
A_\theta(t, r) := - \frac12 \int_0^r \rho(s) s ds 
\end{equation}
When it is important to indicate the dependence upon the density function $%
\rho$, we write $A_\theta^{(\rho)}(t, r)$ for $A_\theta(t, r)$. 
Recall
\begin{equation}  \label{A0elab}
A_0^{(\rho_1, \rho_2)}(t, r) = - \int_r^\infty A_\theta^{(\rho_1)}(s)
\rho_2(s) \frac{ds}{s}
\end{equation}
where $A_\theta^{(\rho_1)}$ is defined using \eqref{Adef}, but with $\rho_1$ in place of $\rho$
in the right-hand side, i.e.,
\[
A_\theta^{(\rho_1)}(t, r) = - \frac12 \int_0^r \rho_1(s) s ds
\]

Define the operators $[\partial_r]^{-1}$, $[r^{-n} \bar{\partial}_r]^{-1}$,
and $[r \partial_r]^{-1}$ acting on radial functions by 
\begin{equation*}
\begin{split}
[\partial_r]^{-1} f(r) &= - \int_r^\infty f(s) ds, \quad \quad [r^{-n} \bar{%
\partial}_r]^{-1} f(r) = \int_0^r f(s) s^n ds \\
[r \partial_r]^{-1} f(r) &= - \int_r^\infty \frac1s f(s) ds
\end{split}%
\end{equation*}
Then it follows by a direct argument that 
\begin{align}
\| [r \partial_r]^{-1} f \|_{L^p} &\lesssim_p \| f \|_{L^p}, \quad \quad 1
\leq p < \infty  \label{direct1} \\
\| r^{-n-1} [r^{-n} \bar{\partial}_r]^{-1} f \|_{L^p} &\lesssim_p \| f
\|_{L^p}, \quad \quad 1 < p \leq \infty  \label{direct2} \\
\| [\partial_r]^{-1} f \|_{L^2} &\lesssim \| f \|_{L^1}  \label{direct3}
\end{align}
These estimates appear, for instance, in \cite[(1.5)]{BeIoKeTa}, and also find application
in \cite[\S 2]{LiSm13}.
\begin{rem}
In these estimates and those below, we use the Lebesgue measure
on $\R^2$ for all $L^p$ spaces. In particular, for radial functions of $r$,
we essentially adopt the $rdr$ measure.
\end{rem}

\begin{lem}[Elementary bounds for $A$]
The connections coefficients $A_\theta$ and $A_0$, given by \eqref{Adef},
satisfy 
\begin{equation}  \label{Athet}
\| A_\theta \|_{L^\infty_x} \lesssim \| \rho \|_{L^1_x}, \quad \| \frac1r
A_\theta \|_{L^\infty_x} \lesssim \| \rho \|_{L^2_{x}}, \quad \| \frac{%
A_\theta}{r^2} \|_{L^p_x} \lesssim \| \rho \|_{L^p_x} \quad \text{where}
\quad 1 < p \leq \infty
\end{equation}
and 
\begin{equation}  \label{A0}
\| A_0 \|_{L^p_x} \lesssim \| \rho \|_{L^1_x} \| \rho \|_{L^p_x} \quad \text{%
where} \quad 1 \leq p < \infty, \quad \| A_0 \|_{L^\infty_x} \lesssim \|
\rho \|_{L^2_x}^2
\end{equation}
Moreover, $A_\theta^2$ satisfies the bounds 
\begin{equation}  \label{Athet2}
\| \frac{1}{r^2} A_\theta^2 \|_{L^p_x} \lesssim \| \rho \|_{L^1_x} \| \rho
\|_{L^p_x} \quad \text{where} \quad 1 < p \leq \infty, \quad \| \frac{1}{r^2}
A_\theta^2 \|_{L^\infty_x} \lesssim \| \rho \|_{L^2_x}^2
\end{equation}
\end{lem}

\begin{proof}
These estimates are essentially contained in \cite[\S 2]{LiSm13}.

The first inequality of \eqref{Athet} is trivial. The second follows from
Cauchy-Schwarz: 
\begin{equation*}
| A_\theta(t, r) | \lesssim r \left( \int_0^\infty | \rho(s) |^2 s ds
\right)^{1/2}
\end{equation*}
The third is an application of \eqref{direct2} with $n = 1$.

The first inequality of \eqref{A0} follows from the first inequality of %
\eqref{Athet} and from \eqref{direct1}. The second is a consequence of
Cauchy-Schwarz and the third inequality of \eqref{Athet} with $p = 2$.

The first inequality of \eqref{Athet2} follows from the first and third
inequalities of \eqref{Athet}. The second follows from two applications of
the second inequality of \eqref{Athet}.
\end{proof}

\begin{lem}[Weighted estimates]
Let $\frac{1}{p}+\frac{1}{q}=1$ with $1<q<\infty $ and suppose that $\rho
=|\psi |^{2}$ and $\rho _{j}=|\psi _{j}|^{2}$ for $j=1,2$. Then 
\begin{align}
\Vert r^{-2/q}A_{\theta }^{(\rho )}\Vert _{L_{x}^{\infty }}& \lesssim \Vert
\psi \Vert _{\dot{H}_{x}^{1/q}}^{2}  \label{AthetaL2} \\
\Vert r^{-1/q}A_{\theta }^{(\rho )}\Vert _{L_{x}^{\infty }}& \lesssim \Vert
\psi \Vert _{\dot{H}_{x}^{1/q}}\Vert \psi \Vert _{L_{x}^{2}}
\label{AthetaPower}
\end{align}%
and 
\begin{equation}
\Vert r^{1/p}A_{0}^{(\rho _{1},\rho _{2})}\Vert _{L_{x}^{\infty }}\lesssim
\min_{\tau \in S_{2}}\Vert \psi _{\tau (1)}\Vert _{\dot{H}%
_{x}^{1/q}}^{2}\Vert \psi _{\tau (2)}\Vert _{\dot{H}_{x}^{1/p}}\Vert \psi
_{\tau (2)}\Vert _{L_{x}^{2}}  \label{A0Linfinity}
\end{equation}%
where $S_{2}$ denotes the set of permutations on two elements.
\end{lem}

\begin{proof}
To establish \eqref{AthetaPower}, use H\"older's inequality to obtain 
\begin{equation*}
| A_\theta | \lesssim r^{2/q} \| \psi \|_{L^{2p}}^2
\end{equation*}
and then use Sobolev embedding. The estimate \eqref{AthetaL2} follows from
H\"older's inequality, which yields 
\begin{equation*}
| A_\theta | \lesssim r^{1/q} \| r^{-1/q} \psi \|_{L^2_x} \| \psi \|_{L^2_x},
\end{equation*}
and Hardy's inequality.

To prove \eqref{A0Linfinity}, use H\"older to write 
\begin{equation*}
\left| A_0^{(\rho_1, \rho_2)} \right| \lesssim \| r^{-2/q}
A_\theta^{(\rho_1)} \|_{L^\infty_x} \| r^{-1/p} \psi_2 \|_{L^2_x} \| \psi_2
\|_{L^2_x} r^{-1/p}
\end{equation*}
Then, using \eqref{AthetaL2} and Hardy's inequality, we obtain 
\begin{equation*}
\| r^{1/p} A_0^{(\rho_1, \rho_2)} \|_{L^\infty_x} \lesssim \| \psi_1 \|_{%
\dot{H}^{1/q}_x}^2 \| \psi_2 \|_{\dot{H}^{1/p}_x} \| \psi_2 \|_{L^2_x}
\end{equation*}
Finally, we may repeat the argument with the roles of $\psi_1$ and $\psi_2$
reversed.
\end{proof}

\begin{lem}[Bounds for the nonlinear terms]
Suppose that $\rho_j = |\psi_j|^2$ for $j = 1, 2$. Then 
\begin{equation}  \label{Abound}
\| A_0^{(\rho_1, \rho_2)} \Theta \|_{L^2_x} + \| \frac{1}{r^2}
A_\theta^{(\rho_1)}A_\theta^{(\rho_2)} \Theta \|_{L^2_x} \lesssim \| \psi_1
\|_{\dot{H}^{\frac12}_x} \| \psi_2 \|_{\dot{H}^{\frac12}_x} \| \Theta \|_{%
\dot{H}^{\frac12}_x} \min_{\tau \in S_2} \| \psi_{\tau(1)} \|_{\dot{H}%
^{\frac12}_x} \| \psi_{\tau(2)} \|_{L^2_x}
\end{equation}
\end{lem}

\begin{proof}
We start with 
\begin{equation*}
\| A_0^{(\rho_1, \rho_2)} \Theta \|_{L^2_x} \lesssim \| r^{1/2}
A_0^{(\rho_1, \rho_2)} \|_{L^\infty_x} \| r^{-1/2} \Theta \|_{L^2_x}
\lesssim \| r^{1/2} A_0^{(\rho_1, \rho_2)} \|_{L^\infty_x} \| \Theta \|_{%
\dot{H}^{1/2}_x}
\end{equation*}
and then appeal to \eqref{A0Linfinity} with $p = q = 2$.

Similarly, 
\begin{equation*}
\begin{split}
\| \frac{1}{r^2} A_\theta^{(\rho_1)} A_\theta^{(\rho_2)} \Theta \|_{L^2_x}
&\lesssim \| r^{-1} A_\theta^{(\rho_1)} \|_{L^\infty_x} \| r^{-1/2}
A_\theta^{(\rho_2)} \|_{L^\infty_x} \| r^{-1/2} \Theta \|_{L^2_x} \\
&\lesssim \| \psi_1 \|_{\dot{H}^{1/2}_x}^2 
\| \psi_2 \|_{\dot{H}^{1/2}_x} \| \psi_2 \|_{L^2_x} \| \Theta
\|_{\dot{H}^{1/2}_x}
\end{split}%
\end{equation*}
where we have used \eqref{AthetaPower} and \eqref{AthetaL2} with $p = q = 2$
and Hardy's inequality. Finally, we may repeat the estimate but with the
roles of $\psi_1$ and $\psi_2$ reversed.
\end{proof}

Now we introduce (see \eqref{a-def} to compare) 
\begin{equation}  \label{abilindef}
a_{\rho_1, \rho_2}(t, r) := A_0^{(\rho_1, \rho_2)}(t, r) + \frac{1}{r^2}
A_\theta^{(\rho_1)}(t, r) A_\theta^{(\rho_2)}(t, r)
\end{equation}
For the definitions of the terms on the right-hand side, see the equations
and comments from \eqref{Adef} to \eqref{A0elab}.

\begin{cor}
\label{Corollary:Key2}Suppose $\rho _{j}=|\psi _{j}|^{2}$ for $j=1,2$. Then 
\begin{equation}  \label{mainnest}
\begin{split}
\Vert a_{\rho _{1},\rho _{2}}\psi _{3}\Vert _{L_{x}^{2}}\lesssim \| \psi_1
\|_{\dot{H}_x^{\frac12}} \| \psi_2 \|_{\dot{H}_x^{\frac12}} \min_{\tau \in
S_3} \| \psi_{\tau(1)} \|_{\dot{H}^{\frac12}_x} \| \psi_{\tau(2)} \|_{\dot{H}%
^{\frac12}_x} \| \psi_{\tau(3)} \|_{L^2_x}
\end{split}%
\end{equation}
where $S_3$ denotes the set of permutations on three elements.
\end{cor}

\begin{proof}
For all but two permutations the estimate follows from \eqref{Abound}. To
establish the estimate for the remaining two cases, we need $L^\infty_x$
bounds on $A_0^{(\rho_1, \rho_2)}$ and $\frac{1}{r^2} A_\theta^{(\rho_1)}
A_\theta^{(\rho_2)}$. Using the second estimate of \eqref{Athet} twice and
Sobolev embedding, we obtain 
\begin{equation*}
\| \frac{1}{r^2} A_\theta^{(\rho_1)} A_\theta^{(\rho_2)} \|_{L^\infty_x}
\leq \| \frac1r A_\theta^{(\rho_1)} \|_{L^\infty_x} \| \frac1r
A_\theta^{(\rho_2)} \|_{L^\infty_x} \lesssim \| \psi_1 \|_{L^4_x}^2 \|
\psi_2 \|_{L^4_x}^2 \lesssim \| \psi_1 \|_{\dot{H}^{1/2}_x}^2 \| \psi_2 \|_{%
\dot{H}^{1/2}_x}^2
\end{equation*}
To bound $A_0^{(\rho_1, \rho_2)}$, we proceed similarly to as was done in
the second estimate of \eqref{A0} and \eqref{A0Linfinity}. In particular,
invoking \eqref{AthetaPower} with $q = 2$ and Hardy, we obtain 
\begin{equation*}
\| A_0^{(\rho_1, \rho_2)} \|_{L^\infty_x} = \| \int_r^\infty s^{-1}
A_\theta^{(\rho_1)} s^{-1} |\psi_2|^2 s ds \|_{L^\infty_x} \lesssim \|
r^{-1} A_\theta^{(\rho_1)} \|_{L^\infty_x} \| r^{-1/2} \psi_2
\|_{L^\infty_x}^2 \lesssim \| \psi_1 \|_{\dot{H}^{1/2}_x}^2 \| \psi_2 \|_{%
\dot{H}^{1/2}_x}^2
\end{equation*}
\end{proof}

\begin{rem}
From the proofs of these estimates we see that the limiting factor in
lowering the regularity of the unconditional uniqueness result lies in the
interaction part, which requires $s = 2/3$ rather than the $s = 1/2$
required for the nonlinear part. By using negative-regularity Sobolev
spaces, \cite{HoTaXi14} lowers the regularity required for the interaction
part. Such a procedure does not seem to work, at least directly, for the
problem at hand. This is because one would need to obtain the same negative
order Sobolev index in the right-hand side of \eqref{mainnest} for the
purpose of moving the term arising from controlling the nonlinear part back
over to the left-hand side (see the argument following the proof of Theorem %
\ref{Thm:estimate for IP}).
\end{rem}

\textbf{Acknowledgments.} The authors thank the referee for
a careful reading of the manuscript and for helpful suggestions for improving
the readability of the paper.






\end{document}